\documentclass[12pt,a4paper,reqno]{amsart}
\usepackage[utf8]{inputenc}
\usepackage[english]{babel}
\usepackage{lmodern}

\usepackage[T1]{fontenc}
\usepackage{amsmath}
\usepackage{amsfonts}
\usepackage{amssymb}
\usepackage{amsthm}

\usepackage{titling}
\usepackage{graphicx}
\usepackage{hyperref}
\usepackage{subcaption}
\usepackage{caption}
\usepackage{color}
\usepackage{calrsfs}
\usepackage{stmaryrd}
\usepackage{multirow}
\usepackage{array}
\usepackage[left=2cm,right=2cm,top=2cm,bottom=2cm]{geometry}

\usepackage{mathtools, stmaryrd}
\usepackage{xparse} \DeclarePairedDelimiterX{\Iintv}[1]{\llbracket}{\rrbracket}{\iintvargs{#1}}
\NewDocumentCommand{\iintvargs}{>{\SplitArgument{1}{,}}m}
{\iintvargsaux#1} %
\NewDocumentCommand{\iintvargsaux}{mm} {#1\mkern1.5mu,\mkern1.5mu#2}

\date{}
\newtheorem{theoreme}{Theorem}
\newtheorem{lemme}{Lemma}
\newtheorem{proposition}{Proposition}
\newtheorem{algo}{Algorithm}
\newtheorem{definition}{Definition}
\newtheorem{remark}{Remark}

\DeclareMathOperator*\bigcircop{\bigcirc}

\title{Approximation by orthonormal polynomials associated with even exponential weights}

\author{Bastien Grosse $^{1}$}
\address{$^1$ Nantes Université, Laboratoire de mathématiques Jean Leray, LMJL, UMR 6629, 2 rue de la Houssinière, BP 92208 F-44322 NANTES Cedex 3 }

\usepackage[foot]{amsaddr}
\numberwithin{equation}{section}
\numberwithin{theoreme}{section}
\numberwithin{lemme}{section}
\numberwithin{proposition}{section}
\numberwithin{algo}{section}
\numberwithin{definition}{section}
\numberwithin{remark}{section}
\begin{document}
\maketitle

\begin{abstract}
In this paper, we prove a quantitative approximation result by orthonormal polynomials associated to an exponential weight of the form $e^{-\phi}$, where $\phi$ is an even polynomial with positive leading coefficient. This result is a consequence of a recursion relation for the orthonormal polynomials and of the strong Poincaré inequality. Simulations are provided at the end of the article, on smooth, non-smooth functions as well as in the Gaussian and the double well case.
\end{abstract}

\tableofcontents

\newpage
\section{Introduction}

\subsection{Context}

This paper is devoted to the approximation of functions by orthonormal polynomials associated with a weight $\rho(x) = e^{-\phi(x)}$, where $\phi$ is an even, nonconstant polynomial.Let us explain the context and the issues raised by this problem.  For this, let us take $f$ a function in the space $L^{2}(\rho)$, i.e $f$ is such that $f\sqrt{\rho}$ belongs to $L^{2}(\mathbb{R})$. If  $(\tilde{P}_{l})_{l\in\mathbb{N}}$ is the basis of orthonormal polynomials   associated with the weight $\rho$, and $N\in\mathbb{N}^{*}$, then $f$ can be projected on the finite-dimensionnal space $V_{N}:= Span\{\tilde{P}_{n} \ | \ n=0,1,...,N \}$. The orthogonal projection of $f$ on $V_{N}$ is denoted by $\pi_{V_{N}}f$. Following \cite{akhiezer}, the space  $Span\{\tilde{P}_{n} \ | \ n\in\mathbb{N} \}$ is total in $L^{2}(\rho)$, thus  the projection $\pi_{V_{N}}f$ converges in $L^{2}(\rho)$ norm toward $f$.  One may ask the following question:

$$
\mbox{\textbf{(Q)} At which speed does $\|f - \pi_{V_{N}}f  \|_{L^{2}(\rho)}$ converge toward 0 with respect to $N$ ?}
$$

\ \\
This question admits partial answers. The simpliest case for which a precise answer is known is the Hermite case ($\phi(x)=\dfrac{1}{2}(x^{2}+\ln(2\pi))$. The result, which has been proven in \cite{guo}, \cite{bessemoulinchatard} for the closely related Hermite functions, is proven here in section 3 for Hermite polynomials. It states that for all $k\in\mathbb{N}^{*}$, there exists a positive constant $ \Lambda_{k}$  such that, for all $f\in H^{k}(\rho)$, 

$$
\| f - \pi_{V_{N}}f \|_{L^{2}(\rho)} \leq \Lambda_{k} \| f\|_{H^{k}(\rho)} \dfrac{1}{N^{\frac{k}{2}}} .
$$

This result generalizes for the so called \textit{Freud weights}. Following \cite{lubinsky2007survey}, a  weight $\rho=e^{-\phi}$ is called  a Freud weight if $\phi$ satisfies the definition.

\begin{definition}[Freud weight]
The function $e^{-\phi}$ defines a Freud weight if :

\begin{itemize}
\item $\phi$ is even ;
\item $\phi'$ exists and $\phi'>0$ on $]0,+\infty[$ ;
\item $x\phi'(x)$ is strictly increasing with right limit 0 at 0 ;
\item there exist $\lambda, A,B>1$, and $C>0$ such that 
$$
A \leq \dfrac{\phi'(\lambda x)}{\phi'(x)} \leq B \ , \ \forall x\geq C .
$$
\end{itemize}
\end{definition}

Notice that $\phi$ is not necessarily a polynomial. For Freud weights, the following theorem gives a bound for the projection error.

\begin{theoreme}[\cite{LEVIN1987}]
Let  $r\geq 1$, $n\in\mathbb{N}$ and let $b_{n}$ be the unique positive solution of the equation 

$$
n = \dfrac{2}{\pi} \int_{0}^{1} \dfrac{(b_{n}t)\phi'(b_{n}t)}{\sqrt{1-t^{2}}} dt .
$$
There exists a constant $K>0$ such that for all functions $f : \mathbb{R} \mapsto \mathbb{R} $ having $r-1$ continuous derivatives, and such that $f^{(r)}$ is absolutely continuous and $f^{(r)}\in  L^{2}(\rho)$: 

$$
\|f - \pi_{V_{N}}f  \|_{L^{2}(\rho)}  \leq K \left(\dfrac{b_{n}}{n}\right)^{r}  \| f^{(r)} \|_{L^{2}(\rho)} .
$$

\end{theoreme}

The positive number $b_{n}$ can be computed in some cases. For example, for the simple weight $\rho(x) = e^{-|x|^{\alpha}}$ with $\alpha>0$, $b_{n}$ is proportionnal to $n^{\frac{1}{\alpha}}$, thus the order of convergence is $r(\frac{1}{\alpha}-1)$. So far, the Freud class is the most general class of weights of the form $e^{-\phi}$, with $\phi$ having at most polynomial growth, for which there exists an approximation result (see \cite{lubinsky2007survey}). In general, it has been proven in \cite{Lubinsky2006}, that \textbf{(Q)} admits an answer for a large class of exponential weights. It is shown that if  $\phi$ is differentiable, and $\lim\limits_{x\to+\infty}\phi'(x)= +\infty$, $\lim\limits_{x\to-\infty}\phi'(x)= -\infty$, then there exists a sequence  $(\eta_{n})_{n\in\mathbb{N}}$ with limit 0 such that, for all $f\in H^{1}(\rho)$ absolutely continuous, we have

\begin{equation}
\|f - \pi_{V_{N}}f  \|_{L^{2}(\rho)} \leq \eta_{N} \| f'  \|_{L^{2}(\rho)}.
\label{equation/eta}
\end{equation}

Note that the sequence $(\eta_{n})_{n\in\mathbb{N}}$ is explicit only in some particular cases. To sum up, the question \textbf{(Q)} admits partial answers, depending on the form of the weight.
\ \\
\ \\
 Our goal is to give an answer to \textbf{(Q)} under the following light assumption \textbf{(H)}: 

$$
\textbf{(H)}: \ \  \mbox{ $\phi$ is  an even nonconstant polynomial of positive leading coefficient.   } 
$$

Thus, we can write $\phi$ in the standard basis :

\begin{equation}
\phi(x) = \sum_{p=0}^{m} v_{p} x^{2p} \label{equation/phi}
\end{equation} 

with $m\geq 1$. Notice that $e^{-\phi}$ is integrable under \textbf{(H)}, and that \textbf{(H)} does not require $\phi$ to be monotonic on $]0,+\infty[$, contrary to the Freud class. The main result (Theorem \ref{theoreme/mainresult}) is to our knowledge the first one concerning, non monotonic potentials on $\mathbb{R}^{+}$. It gives in particular a result for double well potentials.

\subsection{Notations and main result}

Throughout the paper, denote  

\begin{equation}
\rho: = e^{-\phi}
\end{equation}

the exponential weight where $\phi$ verifies \textbf{(H)}. Let $L^{2}(\rho)$ be the space of square integrable functions with respect to the measure $\rho(x)dx$ on $\mathbb{R}$. Let $H^{k}(\rho)$ be the Sobolev space consisting  of  functions $f$ such that
$$
 \sum_{i=0}^{k} \| f^{(i)} \|_{L^{2}(\rho)} < \infty,
$$

where $f^{(i)}$ is the $i$-th derivative of $f$ in the  distributional sense. The set $C_{c}^{\infty}(\mathbb{R})$ consists of  smooth functions with compact support. Recall that it is dense in $H^{k}(\rho)$ for all $k\in\mathbb{N}$. The space $\mathbb{R}_{n}[X]$ denotes the vector space of polynomials with real coefficients and with degree less than $n$. Polynomial will be identified with the associated polynomial function.
Let us now introduce more precisely the family of polynomials $(P_{n})_{n\in\mathbb{N}}$ mentionned in the beginning of the paper. For this, consider $(P_{n})_{n\in\mathbb{N}}$ as the sequence of monic (i.e with leading coefficient equal to 1) orthogonal polynomials of degree $n$ with respect to the weight $\rho$. Recall that this sequence can be built explicitely by applying the Gram-Schmidt orthogonalization process to the monomials $(x^{n})_{n\in\mathbb{N}}$.

Let $(\tilde{P}_{n})_{n\in\mathbb{N}}$ the sequence of orthonormal polynomials with respect to the weight $\rho$. They can be defined by normalization of the orthogonal polynomials, that is,

$$
\tilde{P}_{n} = \dfrac{P_{n}}{\| P_{n} \|_{L^{2}(\rho)}}, \ n \in\mathbb{N}.
$$

Equivalently, they can be defined via  the following recursion formula : 

\begin{equation}
\left\lbrace
\begin{array}{ccl}
x\tilde{P}_{n}(x) & = & a_{n+1} \tilde{P}_{n+1}(x) + a_{n} \tilde{P}_{n-1}(x), \ \ \   n\geq 1, \\
\tilde{P}_{0} &  = &  \dfrac{1}{a_{0}}, \\
\tilde{P}_{-1} &  = &  0 .
\end{array}
\right.
\label{equation/recurrence_polynome_2}
\end{equation}

In the preceeding formula, following  \cite{gautschi}, the coefficients  $a_{n}$ are defined by

\begin{equation}
\left\lbrace
\begin{array}{ccl}
a_{0} & = & \sqrt{\int_{\mathbb{R}}\rho(t) dt}, \\
a_{n} & = & \sqrt{\dfrac{\int_{\mathbb{R}}  P_{n}(t)^{2} \rho(t)dt}{\int_{\mathbb{R}} P_{n-1}(t)^{2} \rho(t) dt}}, \ \  n\geq 1.
\end{array}\right.
\label{equation/a_n}
\end{equation}

Remember that the sequence $(\tilde{P}_{n})_{n\in\mathbb{N}}$ is an Hilbert basis of $L^{2}(\rho)$. For a function $f\in L^{2}(\rho)$, the orthogonal projection $\pi_{V_{N}}f$   on $V_{N}$ is then defined by

$$
\pi_{V_{N}}f = \sum_{k=0}^{N}\left( \int_{\mathbb{R}} f \tilde{P}_{k} \rho dy\right)  \tilde{P}_{k}.
$$

In the sequel, we will refer to the case when $\phi(x) = \dfrac{1}{2}(x^{2}+\ln(2\pi))$ as the \textit{Hermite case}, and when $\phi(x) =(x-1)^{2}(x+1)^{2}$ as the \textit{double well case}, the later being of main interest in modern applications.

We are know ready to state our main result:

\begin{theoreme}
Let $\phi$ be a  polynomial satisfying \textbf{(H)}, and $k>m$. Let $(\gamma_{n})_{n\in\mathbb{N}}$ be the sequence defined recursively by:

$$\left\lbrace
\begin{array}{lll}
\gamma_{0} & = & 1 ,\\
\gamma_{n+1} & = & 2^{\gamma_{n}+1}+\gamma_{n}+2, \ \  n\in\mathbb{N}.
\end{array}\right.
$$

 There exists  $N_{0}\in\mathbb{N}$ and a constant $\Lambda_{\phi,k}>0$ depending only on $\phi$ and $k$, such that for all $N\geq N_{0}$ and $f\in H^{\gamma_{k}}(\rho)$, we have the bound

$$
\| f - \pi_{V_{N}}f \|_{L^{2}(\rho)}^{2} \leq \Lambda_{\phi,k} \| f\|_{H^{\gamma_{k}}(\rho)}^{2} \dfrac{1}{N^{\frac{k}{m}-1}} .
$$
\label{theoreme/mainresult}
\end{theoreme}

Let us make some comments on the result. In the definition of Freud weights, the second point forbids potentials which are non monotonic on $\mathbb{R}^{+}$. Thus, Theorem \ref{theoreme/mainresult} supplements the existing result when $\phi$ is a  polynomial. Our theorem \ref{theoreme/mainresult} requires high regularity on $f$. Hence, it fails to identify the  sequence $(\eta_{n})_{n\in\mathbb{N}}$ in (\ref{equation/eta}). It is anyway quantitative.  The convergence rate in Theorem \ref{theoreme/mainresult} depends linearly on $k$ but the regularity needed grow at least exponentially with $k$. It is a drawback and  our method of proof is far from being optimal. In our proof, the main obstacle is the form of the differential equations satisfied by the orthonormal polynomials as soon as $deg(\phi)>2$ (see Proposition \ref{proposition/formedescoeff}). The fact that the coefficients are rational fractions of polynomials with high degree is a huge difficulty.

\subsection{Outline of the paper}

In Section \ref{section/2} we state different results about orthonormal polynomials and  weighted Sobolev spaces. We start in Subsection \ref{subsection/2.1} by studying a family of differential equations satisfied by the orthonormal polynomials $(\tilde{P}_{n})_{n\in\mathbb{N}}$. The main issue is to find asymptotics for the polynomial coefficients of these equations. We rely heavily on Magnus's result on the asymptotics of the recursion coefficients (see \cite{MAGNUS198665}). Then, we prove in Subsection \ref{subsection/2.2} a generalization of the strong Poincaré inequality (\cite{herau}) for the weighted Sobolev spaces $H^{r}(\rho)$.

In Section \ref{section/3}, we give the proof of our main result. We start by proving the approximation result in the Hermite case in Subsection \ref{subsection/3.1}. While it has pedagogical interest, it gives insight on the proof of the general case which is detailed in Subsection \ref{subsection/3.2}.


Central to our approach is to estimate the coefficients of a sufficiently smooth function $f\in L^{2}(\rho)$ in the basis  $(\tilde{P}_{n})_{n\in\mathbb{N}}$. A key identity is  

\begin{equation*}
\int_{\mathbb{R}} f \tilde{P}_{n} \rho dx = \int_{\mathbb{R}} \left( \bigcircop_{j=1}^{k} \hat{L}_{n+k-j+1} \right)(f) \tilde{P}_{n+k} \rho dx, \ n,k\in\mathbb{N}
\end{equation*}

where the differential operators $(\hat{L}_{p})_{p\in\mathbb{N}}$ steem from the differential equation verified by the orthonormal polynomials (\ref{equation/equadiff}) and $ \bigcircop_{j=1}^{k} \hat{L}_{n+k-j+1}$ is the composition of such operators, from the left to right. The cornerstone of the proof consists in estimating the operator norms by using the strong Poincaré inequality and an induction procedure.

In section 4, we first discuss the numerical computation of orthonormal polynomials. Then, we perform tests on smooth  and on less regular functions to confirm the main result.

\ \\
\textbf{Aknowledgments.} The author would like to thank Mehdi Badsi and Frédéric Hérau for their numerous comments on all the aspects of this work.

\section{Preliminary results}
\label{section/2}

\subsection{Properties of the orthonormal polynomials} 
\label{subsection/2.1}

In the following, we  use a family of differential equations satisfied by the sequence  $(\tilde{P}_{k})_{k\in\mathbb{N}}$. A proof was provided in \cite{YangChen}, and identify the polynomials $A_{n}$ and $B_{n}$ appearing in the next statement. We give the statement and its proof below.

\begin{proposition}
For all $n\in\mathbb{N}$, we have the following differential equation:

\begin{equation}
\left( \dfrac{d}{dx} + B_{n}(x) \right) \tilde{P}_{n}(x)  = A_{n}(x) \tilde{P}_{n-1}(x),
\label{equation/equadiff}
\end{equation}

where the functions $A_{n}$ and $B_{n}$ are polynomials given by:

$$
A_{n}(x) = a_{n} \int_{\mathbb{R}} \dfrac{\phi'(x)-\phi'(y)}{x-y} \tilde{P}_{n}^{2}(y) \rho(y) dy,
$$

$$
B_{n}(x) = a_{n} \int_{\mathbb{R}} \dfrac{\phi'(x)-\phi'(y)}{x-y} \tilde{P}_{n}(y) \tilde{P}_{n-1}(y)\rho(y) dy,
$$

and $a_{n}$ is given by (\ref{equation/a_n}).
The adjoint of $L_{n} :=  \dfrac{d}{dx} + B_{n}$ in $L^{2}(\rho)$ is given formally by:

$$
L_{n}^{*} :=  -\dfrac{d}{dx} + B_{n} + \phi'.
$$ 
\label{proposition/equadiff}

\end{proposition}

\begin{proof}
The polynomial $\tilde{P}_{n}'$ is of degree $n-1$ so it can be projected on the orthonormal basis  $(\tilde{P}_{k})_{0\leq k\leq n-1}$ of $\mathbb{R}_{n-1}[X]$:

$$
\tilde{P}_{n}' = \sum_{k=0}^{n-1} c_{k,n} \tilde{P}_{k}.
$$

By an integration by parts, and since $\tilde{P}_{n} \perp \mathbb{R}_{n-1}[X]$:

\begin{eqnarray*}
c_{n,k} & = & \int_{\mathbb{R}} \tilde{P}_{n}'(y)\tilde{P}_{k}(y) \rho(y) dy \\
  		& = & - \int_{\mathbb{R}} \tilde{P}_{n}(y) (\tilde{P}_{k}'(y)-\phi'(y)\tilde{P}_{k}(y)) \rho(y) dy  \\
  		  		& = & \int_{\mathbb{R}} \tilde{P}_{n}(y) \phi'(y)\tilde{P}_{k}(y) \rho(y) dy. 
\end{eqnarray*}

Hence,

$$
\tilde{P}_{n}'(x) = \int_{\mathbb{R}} \left(\sum_{k=0}^{n-1} \tilde{P}_{k}(x) \tilde{P}_{k}(y) \right)  \tilde{P}_{n}(y) \phi'(y) \rho(y) dy.
$$

Notice that, since $\tilde{P}_{n} \perp \mathbb{R}_{n-1}[X]$, one has the following identity:

$$
 \int_{\mathbb{R}} \left(\sum_{k=0}^{n-1} \tilde{P}_{k}(x) \tilde{P}_{k}(y) \right)  \tilde{P}_{n}(y)  \rho(y) dy =  \sum_{k=0}^{n-1} \tilde{P}_{k}(x)\int_{\mathbb{R}} \tilde{P}_{k}(y) \tilde{P}_{n}(y)  \rho(y) dy = 0.
$$

Combining the two last identities yields

$$
\tilde{P}_{n}'(x) = \int_{\mathbb{R}} \left(\sum_{k=0}^{n-1} \tilde{P}_{k}(x) \tilde{P}_{k}(y) \right)  \tilde{P}_{n}(y) (\phi'(y)-\phi'(x)) \rho(y) dy.
$$

Finally, the Christoffel-Darboux formula (\cite{gautschi}, Theorem 1.32) implies that

$$
\sum_{k=0}^{n-1} \tilde{P}_{k}(x)\tilde{P}_{k}(y) = a_{n} \dfrac{\tilde{P}_{n-1}(y)\tilde{P}_{n}(x) -  \tilde{P}_{n-1}(x)\tilde{P}_{n}(y) }{x-y}.
$$

This establish the differential equation. 

Next, let $u,v$ be two smooth functions. By integration by part, 

\begin{eqnarray*}
\int_{\mathbb{R}} (L_{n}u) v \rho dx & = &  \int_{\mathbb{R}} B_{n} u  v \rho dx   +  \int_{\mathbb{R}} (\dfrac{d}{dx}u)  v \rho dx \\
									& = &  \int_{\mathbb{R}} B_{n} u  v \rho dx   - \int_{\mathbb{R}} u  (\dfrac{d}{dx}v - v \phi'   ) \rho dx \\
									& = &  \int_{\mathbb{R}} u (-\dfrac{d}{dx}v + v \phi' + B_{n} v)   \rho dx.
\end{eqnarray*} 

This gives formally the adjoint of $L_{n}$.
 
\end{proof}

It is instructive to explicit (\ref{equation/equadiff}) in the Hermite case. An easy computation gives that $A_{n} = a_{n}$, and that $B_{n}=0$. Since $a_{n}=\sqrt{n}$ for $n\geq 1$ and  $a_{0}=1$, we get

$$
\tilde{P}_{n}' = \sqrt{n} \tilde{P}_{n-1},
$$

$$
\tilde{P}_{0}' = 0.
$$

This equation is a classical result. The Hermite case is the simplest one for which computations are  explicit. Notice that in this case $A_{n}$ is positive for all $n\in\mathbb{N}$. In the sequel, we will need $A_{n}$ to be a positive polynomial. A trivial case in which all the $A_{n}$ are positive is when $\phi$ is a strictly convex potential. Indeed, the mean-value theorem applied to $\phi'$ shows that

$$
 \dfrac{\phi'(x)-\phi'(y)}{x-y} \tilde{P}_{n}^{2}(y) \rho(y)
$$

is nonnegative for all $x,y\in\mathbb{R}$. Moreover, this integrand does not vanish, so that $A_{n}(x)>0$ for every $x\in\mathbb{R}$ and $n\in\mathbb{N}$.

Recall that in this paper, we are interested in general potential, possibly nonconvex. To establish the positivity of $A_{n}$ in the general case, we have to investigate fine properties of the sequence $(\tilde{P}_{n})_{n\in\mathbb{N}}$. An essential property is the asymptotic behaviour of the coefficients $a_{n}$, established by Magnus in  Theorem 6.1 of \cite{MAGNUS198665}:

\begin{theoreme}[\cite{MAGNUS198665}]
Let $\phi$ a polynomial of even degree $2m$ and of leading coefficient $v_{m}>0$. Then

\begin{equation}
a_{n} \sim \left( \dfrac{  (m-1)!^{2}}{  2v_{m}(2m-1)!} n \right)^{1/2m}.
\end{equation}
\label{theoreme/magnus}
\end{theoreme}

An other lemma of \cite{YangChen} will be used in the following.  For the sake of self-containdness, we remind its statement and proof below. 

\begin{lemme}[\cite{YangChen}]
Let $k,n\in\mathbb{N}$ with $n\geq k$.Then $y^{k} \tilde{P}_{n}(y)$ has the form

$$
y^{k} \tilde{P}_{n}(y) = \sum_{r=0}^{k} c_{n,k,r} \tilde{P}_{n+2r-k}(y),
$$

where for all $r\in \{0,...,k\}$,  $c_{n,k,r}$ is a homogenous polynomial of degree $k$ of the numbers  $a_{n-(k-1)+r},a_{n-(k-1)+r+1},...,a_{n+r}$ with positive coefficients.
\label{lemme/lemme1}
\end{lemme}

\begin{proof}
Let us work by induction on $k$. If $k=0$ the statement is trivial. Assume that it is true for a fixed integer $k$, with $k+1\leq n$. The induction relies on the induction relation (\ref{equation/recurrence_polynome_2}):

\begin{eqnarray*}
y^{k+1} \tilde{P}_{n}(y) & = & \sum_{r=0}^{k} c_{n,k,r} y\tilde{P}_{n+2r-k}(y) \\
					    & = &  \sum_{r=0}^{k} c_{n,k,r} (a_{n+2(r+1)-(k+1)}\tilde{P}_{n+2(r+1)-(k+1)}(y) +  a_{n+2r-k}\tilde{P}_{n+2r-(k+1)}(y) ) \\
					    & = &  \sum_{r=1}^{k+1} c_{n,k,r-1} a_{n+2r-(k+1)}\tilde{P}_{n+2r-(k+1)}(y) +  \sum_{r=0}^{k}  a_{n+2r-k} c_{n,k,r} \tilde{P}_{n+2r-(k+1)}(y)  \\
					    & = & c_{n,k,0} a_{n-k} \tilde{P}_{n-(k+1)}(y) + \sum_{r=1}^{k} (c_{n,k,r-1} a_{n+2r-(k+1)}  +   a_{n+2r-k} c_{n,k,r}   )\tilde{P}_{n+2r-(k+1)}(y) \\
					   & +& c_{n,k,k} a_{n+(k+1)} \tilde{P}_{n+(k+1)}(y). 
\end{eqnarray*}

We can now pose 

$$c_{n,k+1,r} = \left\lbrace
\begin{array}{ccc}
c_{n,k,0}a_{n-k} & \ & r=0, \\
c_{n,k,k} a_{n+(k+1)} & \ & r=k+1, \\
c_{n,k,r-1} a_{n+2r-(k+1)}  +   a_{n+2r-k} c_{n,k,r} & \ & 1\leq r \leq k .
\end{array}\right.
$$

Then, 

$$y^{k+1} \tilde{P}_{n}(y) = \sum_{r=0}^{k+1} c_{n,k+1,r} \tilde{P}_{n+2r-k-1}(y).$$

This proves the result.

\end{proof}

We are now ready to prove the positivity of the polynomials $A_{n}$, at least for $n$ large enough.

\begin{proposition}

For all $x\in\mathbb{R}$ and $n\in\mathbb{N}$, the polynomials $A_{n}$ and $B_{n}$ have the expressions 

$$
A_{n}(x) = a_{n}\sum_{l=0}^{m-1} A_{n,l} x^{2l} \ ; \ B_{n}(x) = a_{n}\sum_{l=0}^{m-2} B_{n,l} x^{2l+1}. 
$$

We have that the $A_{n,l}$, $0\leq l\leq m-1$ are positive for $n$  large enough. Moreover, we have the the asymptotic behaviours when $n$ goes to $+\infty$:

$$
A_{n,l} = O_{n}(a_{n}^{2m-2-2l}) \ ; \ |B_{n,l}| = O_{n}(a_{n}^{2m-3-2l}).
$$

Consequently, there exists $N_{0}\in\mathbb{N}$ such that for all $n\geq N_{0}$, the polynomials $A_{n}$ are positive, convex and reach their global minimum at $0$. Moreover:

$$
\left\| \dfrac{1}{A_{n}} \right\|_{L^{\infty}(\mathbb{R})} \leq \dfrac{1}{A_{n}(0)}.
$$

\label{proposition/bornefraction}
\end{proposition}

\begin{proof}

Let us write $A_{n}$ in the standard basis. By using remarkable identity  $$x^{p}-y^{p}=(x-~y)\sum_{k=0}^{p-1-k}x^{k}y^{p-1-k}$$ we get that, recalling equation \ref{equation/phi}:

$$
\dfrac{\phi'(x)-\phi'(y)}{x-y}\tilde{P}_{n}(y) = \sum_{p=1}^{m}2pv_{p} \sum_{k=0}^{2p-2}x^{2p-2-k} y^{k}\tilde{P}_{n}(y).
$$

Next, by mean of  Lemma \ref{lemme/lemme1}, this polynomial in the $y$-variable can be projected  on the basis $(\tilde{P}_{l})_{l\in\mathbb{N}}$. Let's assume that $2m-2\leq n$. Lemma \ref{lemme/lemme1} states that $y^{k} \tilde{P}_{n}$ has a  $\tilde{P}_{n}$ component only if $k$ is even,  and the corresponding coefficient is equal to $c_{n,k,k/2} = \int_{\mathbb{R}} y^{k} \tilde{P}_{n}(y)^{2}\rho(y)dy$.

Consequently,

\begin{eqnarray*}
A_{n}(x)  & = &  a_{n}\int_{\mathbb{R}}  \dfrac{\phi'(x)-\phi'(y)}{x-y}\tilde{P}_{n}(y)  \tilde{P}_{n}(y) \rho(y) dy  \\
		&= & a_{n}\int_{\mathbb{R}} \sum_{p=1}^{m}2pv_{p} \sum_{k=0, \ k \ even}^{2p-2}x^{2p-2-k} y^{k} \tilde{P}_{n}^{2}(y) \rho(y) dy. 
\end{eqnarray*}

 By performing a change of indices and rearanging the sums, we find that

\begin{eqnarray*}
A_{n}(x)  & = &   a_{n}\int_{\mathbb{R}}  \sum_{p=1}^{m}2pv_{p} \sum_{r=0}^{p-1}x^{2p-2-2r} y^{2r} \tilde{P}_{n}^{2}(y) \rho(y) dy \mbox{ \ \ ($k=2r$)}  \\
	 & = & a_{n}\int_{\mathbb{R}}  \sum_{p=1}^{m}2pv_{p} \sum_{l=0}^{p-1}x^{2l} y^{2p-2-2l} \tilde{P}_{n}^{2}(y) \rho(y) dy   \mbox{ \ \ ($l=p-1-r$)}  \\
	 & = & a_{n}\int_{\mathbb{R}}    \sum_{l=0}^{m-1} \sum_{p=1}^{m}2pv_{p} \delta_{\{l\leq p-1\}} y^{2p-2-2l}   x^{2l} \tilde{P}_{n}^{2}(y) \rho(y) dy. 
\end{eqnarray*}

where $\delta_{\{l\leq p-1\}}$ equals 1 if $l\leq p-1$, and 0 otherwise.

Define $Q_{l}(y) :=  \sum_{p=1}^{m}2pv_{p} \delta_{\{l\leq p-1\}} y^{2p-2-2l} $. Remark that $Q_{l}$ is an even polynomial of degree $2m-2-2l$ and of leading coefficient $2mv_{m}>0$. In the standard basis, $A_{n}$ reads then

$$
A_{n}(x) = a_{n}\left(2mv_{m} x^{2m-2} +  \sum_{l=0}^{m-2} \int_{\mathbb{R}} Q_{l}(y)   \tilde{P}_{n}^{2}(y) \rho(y) dy  x^{2l} \right).
$$

Now,  apply  Lemma \ref{lemme/lemme1} to express the integral  $ \int_{\mathbb{R}} Q_{l}(y)   \tilde{P}_{n}^{2}(y) \rho(y) dy $:

$$
\int_{\mathbb{R}} Q_{l}(y)   \tilde{P}_{n}^{2}(y) \rho(y) dy  = \sum_{p=1}^{m} 2pv_{p} \delta_{l\leq p-1} c_{n,2p-2-2l,p-1-l}.
$$

This expression depends on $n$. According to Lemma \ref{lemme/lemme1} and Theorem \ref{theoreme/magnus}, the dominant term is $c_{n,2m-2-2l,m-1-l}$, and the leading coefficient is  $2mv_{m}>0$. Hence, the last integral tends to  $+\infty$ when $n\to \infty$. Therefore, there exists $N_{0}$ such that all the coefficients of $A_{n}$ in the standard basis are  positive. Moreover, as $A_{n}$ is even, $A_{n}$ reaches its global minimum at 0 and is strictly convex. It implies the following inequality for all $x\in\mathbb{R}$ and $n\geq N_{0}$:

$$
0< \dfrac{1}{A_{n}(x)} \leq \dfrac{1}{A_{n}(0)}.
$$

The same process can be applied to  $B_{n}$. We find

$$
B_{n}(x) = a_{n} \sum_{l=0}^{m-2}   \sum_{p=1}^{m}2pv_{p} \delta_{\{l\leq p-2\}} c_{n,2p-2l-3,p-l-2}  x^{2l+1}.
$$

The coefficient in front of $x^{2l+1}$ in the standard basis tends toward $+\infty$ as $a_{n}^{2m-3-2l}$.

\end{proof}

\begin{remark}

All the preceeding computations can be done explicitely on the double well case $\phi(x) =  (x-1)^{2}(x+1)^{2} = x^{4}-2x^{2}+1$. We find that:

$$
A_{n}(x) = 4 a_{n}(x^{2} + a_{n}^{2}+a_{n+1}^{2}-1) \  ; \ \ B_{n}(x) = 4a_{n}^{2}x. 
$$

This shows that $A_{n}$ is positive as long as $a_{n}^{2}+a_{n+1}^{2}>1$.

\end{remark}
\subsection{Functional inequality}
\label{subsection/2.2}
Let us move on to the strong Poincaré inequality for the measure $\rho(x)dx$ proven in e.g. \cite{poincare}.

\begin{theoreme}[Strong Poincaré inequality \cite{poincare}]
There exists a constant $C_{p}>0$ such that for all $f\in H^{1}(\rho)$, 

$$
\int_{\mathbb{R}} (1+\phi^{'2}(x)) |f(x)-<f>|^{2} \rho(x) dx \leq C_{p} \int_{\mathbb{R}} f^{'2}(x) \rho(x) dx 
$$ 

where $<f> := \int_{\mathbb{R}} f \rho dx$ denotes the mean value of $f$ with respect to the measure $\rho(x)dx$.
\label{theoreme/poincare}
\end{theoreme}

This inequality alone guarantees that if $f\in H^{1}(\rho)$, then the product $x^{k}f$ belong to $L^{2}(\rho)$, for $k\leq 2m-1$: regularity implies more integrability. We will generalize this inequality for more regular functions $f$. This will allow us to control  $L^{2}(\rho)$ norms of products between these functions and polynomials up to a certain degree.

\begin{proposition}
There exists a constant $C_{\phi,r}>0$ depending only on $\phi$ and $r\in\mathbb{N}$ such that for all $f\in H^{r}(\rho)$,

$$
\| (\phi^{'})^{r} f \|_{L^{2}(\rho )}^{2} \leq C_{\phi,r} \| f\|_{H^{r}(\rho )}^{2}.
$$
In particular, if $P\in\mathbb{R}_{r(2m-1)-1}[X]$, then there exists a constant $K_{\phi,r,P}>0$ depending only on $P,\phi,r$ such that

$$
\| P f\|_{L^{2}(\rho)}^{2} \leq K_{\phi,r,P} \| f \|_{H^{r}(\rho)}^{2}.
$$ 
\label{proposition/poincaregeneralise}
\end{proposition}

\begin{proof}
We proceed by induction on $r$.

\textbf{Base case :} We have first  to check the property for $r=1$. By the strong Poincaré inequality, if $f\in H^{1}(\rho)$,

\begin{eqnarray*}
\int_{\mathbb{R}} \phi^{'2} f^{2} \rho dx & = & \int_{\mathbb{R}} \phi^{'2} | f-<f>+<f>|^{2} \rho dx \\
 & \leq & 2\int_{\mathbb{R}} \phi^{'2} | f-<f>|^{2} \rho dx + 2\int_{\mathbb{R}} \phi^{'2} \rho dx  <f>^{2} \mbox{ (Young inequality)} \\
 & \leq &  2 c_{p} \| f' \|_{L^{2}(\rho)}^{2} + 2\int_{\mathbb{R}} \phi^{'2} \rho dx  \| f\|_{L^{2}(\rho)}^{2}  \mbox{ (Jensen and Strong Poincaré inequality)} \\
 & \leq & C_{\phi,1} \| f \|_{H^{1}(\rho)}^{2}
\end{eqnarray*}

with $C_{1,\phi}=2c_{p}+2\int_{\mathbb{R}} \phi^{'2} \rho dx $. This shows the result for $r=1$.
\ \\

\textbf{Induction step :} Let us assume that the property is true for a fixed integer $r\geq 1$. Let $f\in H^{r+1}(\rho)$. Trivially, $f\in H^{r}(\rho)$. By assumption, there exists $C_{\phi,r}$ such that 
$ \| \phi^{'r} f \|_{L^{2}(\rho)}^{2} \leq C_{\phi,r} \| f\|_{H^{r}(\rho)}^{2}$.

Since the function $f$ is in $L^{1}_{loc}$, $\phi^{'r}$ is smooth, the product $\phi^{'r}f$ defines a distribution and $(\phi^{'r}f)'=r\phi''\phi^{'r-1}f+\phi^{'r}f'$. This right-hand side is also a function which we prove to belong to $L^{2}(\rho)$ by showing that it is the sum of two functions in $L^{2}(\rho)$. By considerations on the degree of $\phi$ and its derivatives, 

$$
\left| \dfrac{r\phi''(x)(\phi^{'})^{r-1}(x)}{(\phi^{'})^{r}(x)} \right| \xrightarrow[|x|\to +\infty]{} 0,
$$

so there exists $M_{r,\phi}>0$ such that $|r\phi''(x)(\phi^{'})^{r}(x)|\leq |(\phi^{'})^{r}(x)|$ if $|x|\geq M_{r,\phi}$, and $|r\phi''(x)(\phi^{'})^{r}(x)|\leq \|r\phi''(\phi^{'})^{r}\|_{L^{\infty}([-M_{r,\phi},M_{r,\phi}])}$ otherwise. Therefore,

\begin{eqnarray*}
\int_{\mathbb{R}} |r\phi''(\phi^{'})^{r}|^{2} f^{2} \rho dx & \leq &   \|r\phi''(\phi^{'})^{r}\|_{L^{\infty}([-M_{r,\phi},M_{r,\phi}])}  \int_{|x|\leq M_{r,\phi}}f^{2} \rho dx +  \int_{|x|\geq M_{r,\phi}} |(\phi^{'})^{r}|^{2} f^{2} \rho dx \\
& \leq &   \|r\phi''(\phi^{'})^{r}\|_{L^{\infty}([-M_{r,\phi},M_{r,\phi}])}  \int_{\mathbb{R}}f^{2} \rho dx +  \int_{\mathbb{R}} |(\phi^{'})^{r}|^{2} f^{2} \rho dx \\
& \leq &  \|r\phi''(\phi^{'})^{r}\|_{L^{\infty}([-M_{r,\phi},M_{r,\phi}])} \| f\|_{L^{2}(\rho)}^{2} + C_{\phi,r} \| f\|_{H^{r}(\rho)}^{2}.
\end{eqnarray*}

The induction hypothesis was applied in the third line.

By definition, as $f\in H^{r+1}(\rho)$, $f'\in H^{r}(\rho)$. The induction hypothesis can be applied to $f'$ : $\| (\phi^{'})^{r} f' \|_{L^{2}(\rho)}^{2} \leq C_{\phi,r} \| f'\|_{H^{r}(\rho)}^{2}$.

We have proven that $((\phi^{'})^{r}f)'\in L^{2}(\rho)$, hence $(\phi^{'})^{r}f\in H^{1}(\rho)$. The strong Poincaré inequality applied to $(\phi^{'})^{r}f$, and the computations already done for the base case and the induction step give then

\begin{eqnarray*}
\| \phi^{'r+1}f \|_{L^{2}(\rho)}^{2}  & \leq & 2C_{p} \| ((\phi^{'})^{r}f)^{'}  \|_{L^{2}(\rho )}^{2}  + 2 C_{p} \int_{\mathbb{R}} \phi^{'2} \rho   \| (\phi')^{r} f\|_{L^{2}(\rho)}^{2} \\
  & \leq & 2C_{p} \|r\phi''(\phi^{'})^{r-1}f+(\phi^{'})^{r}f' \|_{L^{2}(\rho )}^{2}  + 2 c_{p} \int_{\mathbb{R}} \phi^{'2} \rho dx   \| (\phi')^{r} f\|_{L^{2}(\rho)}^{2} \\
    & \leq & 4C_{p} \|r\phi''(\phi^{'})^{r-1}f \|_{L^{2}(\rho)}^{2} + 4C_{p} \| (\phi^{'})^{r}f' \|_{L^{2}(\rho)}^{2}  + 2 C_{p} \int_{\mathbb{R}} \phi^{'2} \rho  dx \| (\phi')^{r} f\|_{L^{2}(\rho)}^{2} \\
     & \leq  & 4C_{p} \|r\phi''(\phi^{'})^{r}\|_{L^{\infty}([-M_{r,\phi},M_{r,\phi}])} \| f\|_{L^{2}(\rho)}^{2} + 4C_{p} C_{\phi,r} \| f\|_{H^{r}(\rho)}^{2}+ 4C_{p} \| (\phi^{'})^{r}f' \|_{L^{2}(\rho)}^{2}  \\
      & + &  2 C_{p} \int_{\mathbb{R}} \phi^{'2} \rho dx  \| (\phi')^{r} f\|_{L^{2}(\rho)}^{2} \\
      & \leq & C_{\phi,r+1} \| f\|_{H^{r+1}(\rho)},
\end{eqnarray*} 

where $C_{\phi,r+1}$ is a constant depending only on $\phi$ and $r+1$. Notice that we used the induction hypothesis applied to $f$ and $f'$ at the last line. This proves the property at rank $r+1$.\ \\

\textbf{Conclusion :} By induction, the property is true for all $r\geq 1$.\ \\

To show the last point of Proposition \ref{proposition/poincaregeneralise}, it suffices to notice that for any polynomials $P$ such that $deg(P)\leq deg(\phi^{'r})$, one has that there exists $K>0$ such that

$$
\left| \dfrac{P(x)}{(\phi^{'})^{r}(x)} \right| <K
$$

 for $|x|$ large enough. Hence, there exists $M_{P,\phi}>0$ such that $|P(x)|\leq K |(\phi^{'})^{r}(x)|$ if $|x|\geq M_{P,\phi}$, and $|P(x)|\leq \|P\|_{L^{\infty}([-M_{P,\phi},M_{P,\phi}])}$ otherwise. By the same process used earlier in the induction step, we get

$$
\| Pf\|_{L^{2}(\rho)}^{2} \leq K \| (\phi^{'})^{r}f \|_{L^{2}(\rho)}^{2} +  \|P\|_{L^{\infty}([-M_{P,\phi},M_{P,\phi}])}^{2} \| f \|_{L^{2}(\rho)}^{2}.
$$

By this last inequality and by the first part of the proposition, there exists  $K_{\phi,r,P}>0$ such that

$$
\| P f\|_{L^{2}(\rho )}^{2} \leq K_{\phi,r,P} \| f \|_{H^{r}(\rho )}^{2}.
$$ 
\end{proof}
\newpage

\section{Convergence of the projection error}
\label{section/3}
Throughout this section, $f$ is a function of $C^{\infty}_{c}(\mathbb{R})$.

\subsection{Hermite case}
\label{subsection/3.1}
We remind here the proof of the spectral convergence theorem for the Hermite polynomials. We have already seen that in this case, operators $L_{n}$ and $L_{n}^{*}$ defined in Proposition \ref{proposition/equadiff} are independant of $n$, as $B_{n}=0$. This motivates the simpler notation:

$$
Lf = f'\ ; \ L^{*}f = -f'+xf.
$$

This specificity leads to a  simpler and more efficient proof than in the general setting. We  need to know  the form of the iterated operator $(L^{*})^{k}$.

\begin{proposition}
For all $k\geq 1$, 

$$(L^{*})^{k}f = \sum_{j=0}^{k} Q_{k,k-j}(x) f^{(j)}$$

where $Q_{k,k-j}$ is a polynomial of degree $k-j$.

\end{proposition} 

\begin{proof}
Let us reason by induction.\ \\

\textbf{Base case :} The proposition is immediate for $k=1$.\ \\

\textbf{Induction step :} Assume that the proposition is true for a fixed $k\geq 1$. Then

\begin{eqnarray*}
(L^{*})^{k+1}f & = &  -\sum_{j=0}^{k} (Q_{k,k-j}(x) f^{(j)})'+ \sum_{j=0}^{k} xQ_{k,k-j}(x) f^{(j)}  \\
		  & = &  -\sum_{j=0}^{k} Q_{k,k-j}'(x) f^{(j)} +  Q_{k,k-j}(x) f^{(j+1)} + \sum_{j=0}^{k} xQ_{k,k-j}(x) f^{(j)} \\
		  & = &   -\sum_{j=0}^{k} Q_{k,k-j}'(x) f^{(j)} -  \sum_{j=1}^{k+1}  Q_{k,k-j+1}(x) f^{(j)} + \sum_{j=0}^{k} xQ_{k,k-j}(x) f^{(j)}\\
		  & = & (-Q_{k,k}'(x)+xQ_{k,k}(x)) f +   \sum_{j=1}^{k} (- Q_{k,k-j+1}(x)-Q_{k,k-j}'(x)+xQ_{k,k-j}(x)) f^{(j)} \\
		  &  - & Q_{k,0}(x) f^{(k+1)}. 
\end{eqnarray*}

The first polynomial is of degree $k+1$, and the last is of degree  $0$. The polynomial in front of $f^{(j)}$ is of degree $k+1-j$. Hence the property is true for $k+1$

\ \\
\textbf{Conclusion :} By induction, the property is true for all $k\geq 1$.

\end{proof}

We can now prove Theorem \ref{theoreme/mainresult} in the Hermite case.

\begin{proposition}
Let $\rho(x) = \dfrac{e^{-\frac{x^{2}}{2}}}{\sqrt{2\pi}}$ be the weight in the Hermite case and $k\geq 1$.
Then there exists a constant $\Lambda_{k}>0$ depending only on $k$, such that for all $f\in H^{k}(\rho)$,

$$
 \forall N\geq 1, \ \ \| f - \pi_{V_{N}}f \|_{L^{2}(\rho)} \leq \Lambda_{k} \| f\|_{H^{k}(\rho)} \dfrac{1}{N^{\frac{k}{2}}}
$$
\label{proposition/hermitecase}
\end{proposition}

\begin{proof}
The core of the proof is the following equality : for all $k\geq 1$, 
$$
\tilde{P}_{n} =  \left(\prod_{j=1}^{k} \dfrac{1}{\sqrt{n+j}} \right)  L^{k}(\tilde{P}_{n+k}).
$$

By $k$ integrations by part, we find that

$$
\int_{\mathbb{R}} f \tilde{P}_{n} \rho dx =   \left(\prod_{j=1}^{k} \dfrac{1}{\sqrt{n+j}} \right) \int_{\mathbb{R}} f   L^{k}(\tilde{P}_{n+k}) \rho dx = \left(\prod_{j=1}^{k} \dfrac{1}{\sqrt{n+j}} \right) \int_{\mathbb{R}} L^{*k}(f)  \tilde{P}_{n+k} \rho dx.
$$

By the Parseval formula, we then compute

\begin{eqnarray*}
\| f - \pi_{V_{N}}f \|_{L^{2}(\rho)}^{2} & = & \sum_{n=N+1}^{\infty} \left|\int_{\mathbb{R}} f \tilde{P}_{n} \rho dx \right|^{2} \\
 & = &  \sum_{n=N+1}^{\infty} \left(\prod_{j=1}^{k} \dfrac{1}{n+j} \right) \left|\int_{\mathbb{R}} L^{*k}(f)  \tilde{P}_{n+k} \rho dx \right|^{2} \\
 & \leq &  \dfrac{1}{N^{k}} \sum_{n=N+1}^{\infty}   k^{2}\sum_{j=0}^{k}\left|\int_{\mathbb{R}} Q_{k,k-j}(x) f^{(j)} \tilde{P}_{n+k} \rho dx \right|^{2} \mbox{(By Cauchy-Schwarz on $\mathbb{R}^{k+1}$)} \\
 & \leq &  \dfrac{(k+1)^{2}}{N^{k}} \sum_{j=0}^{k}\sum_{n=N+1}^{\infty}   \left|\int_{\mathbb{R}} Q_{k,k-j}(x) f^{(j)} \tilde{P}_{n+k} \rho dx \right|^{2} \\
  & \leq &  \dfrac{(k+1)^{2}}{N^{k}} \sum_{j=0}^{k} \| Q_{k,k-j} f^{(j)} \|_{L^{2}(\rho)}^{2} \mbox{ (By Parseval formula)}\\
  & \leq &  \dfrac{(k+1)^{2}}{N^{k}} \sum_{j=0}^{k} C_{k,j} \|  f^{(j)} \|_{H^{k-j}(\rho)}^{2} \mbox{ (Last point of Proposition \ref{proposition/poincaregeneralise})} \\
  & \leq &  \dfrac{1}{N^{k}}\left( (k+1)^{2}\sum_{j=0}^{k} C_{k,j} \right) \|  f \|_{H^{k}(\rho)}^{2}.
\end{eqnarray*}

Here, $C_{k,j}$ is the best constant in the inequality $\| Q_{k,k-j} f^{(j)} \|_{L^{2}(\rho)}^{2} \leq C_{k,j} \|  f^{(j)} \|_{H^{k-j}(\rho)}^{2}$, given by the last point in Proposition \ref{proposition/poincaregeneralise}. This concludes the proof.
\end{proof}

\subsection{General case}
\label{subsection/3.2}
 We use first the differential equation (\ref{equation/equadiff}) to express $\tilde{P}_{n}$. If $n$ is large enough, $A_{n+1}$ is a positive polynomial. Therefore, we can write

$$
\tilde{P}_{n} = \dfrac{L_{n+1}(\tilde{P}_{n+1})}{A_{n+1}}.
$$

An integration by part in the coefficients of  $f$ in the basis $(\tilde{P}_{l})_{l\in\mathbb{N}}$ gives

$$
\int_{\mathbb{R}} f \tilde{P}_{n} \rho dx = \int_{\mathbb{R}} f  \dfrac{L_{n+1}(\tilde{P}_{n+1})}{A_{n+1}} \rho dx = \int_{\mathbb{R}} L^{*}_{n+1}\left(\dfrac{f}{A_{n+1}}\right) \tilde{P}_{n+1} \rho dx.
$$

In fact, the computation can be iterated and we get for $k\geq 1$, 

\begin{equation}
\int_{\mathbb{R}} f \tilde{P}_{n} \rho dx = \int_{\mathbb{R}} \left( \bigcircop_{j=1}^{k} \hat{L}_{n+k-j+1} \right)(f) \tilde{P}_{n+k} \rho dx 
\label{equation/10}
\end{equation}

where $\hat{L}_{p}$ is a linear differential operator defined by

\begin{equation}
\hat{L}_{p} := L^{*}_{p}\left(\dfrac{.}{A_{p}}\right).
\end{equation}

and

\begin{equation}
\bigcircop_{j=1}^{k} \hat{L}_{n+k-j+1} = \hat{L}_{n+k} \circ \hat{L}_{n+k-1}\circ ...\circ \hat{L}_{n+1}.
\end{equation}

The fact that the operators depends now on $n$  will cause us trouble  to generalize the proof for the Hermite case. We will use a different method instead, at the price of more regularity.
The first step toward the result is to understand the form of the derivative $(\hat{L}_{p}(f))^{(j)}$.

\begin{proposition}
Let $p$ be an integer large enough so that $A_{p}$ is positive. For all $j\in\mathbb{N}$, the derivative $ \hat{L}_{p}(f)^{(j)}$  has the following form:
 \[ \hat{L}_{p}(f)^{(j)}(x) = \sum_{k=0}^{j+1} \dfrac{Q_{k,j,p}(x)}{A_{p}^{2^{j+1-k}}(x)} f^{(k)}(x) .\]

The polynomials $Q_{k,j,p}$ satisify the three properties:

\begin{enumerate}
\item[1)] $Q_{j+1,j,p}=-1$;
\item[2)] $Q_{k,j,p}$ are polynomials of degree at most  $2^{j+1-k}(2m-2)+1$;
 \item[3)]  $Q_{k,j,p}(x) = a_{p}^{2^{j+1-k}(2m-1)-1} \tilde{Q}_{k,j,p}(x)$ where $\tilde{Q}_{k,j,p}$ is a polynomial whose coefficients are bounded with respect to $p$.
\end{enumerate}
\label{proposition/formedescoeff}
\end{proposition}

\begin{proof}
Let us work by induction on $j$. 

\textbf{Base case :} We check the different points for $j=0$. Remember that the polynomials $A_{p}$ and $B_{p}$, which appear in the differential equation (\ref{equation/equadiff}), are studied in proposition \ref{proposition/bornefraction}. For $j=0$, we have that:

$$
\hat{L}_{p}(f) = - \dfrac{1}{A_{p}} f' +  \dfrac{B_{p}A_{p}+\phi'A_{p}+A'_{p}}{A_{p}^{2}} f  .
$$ 

We now prove the three points of Proposition \ref{proposition/formedescoeff} using this expression.

\begin{enumerate}
\item[1)] It is straighforward.

\item[2)] By the definitions of $A_{p}$, $B_{p}$, the polynomial $B_{p}A_{p}+\phi'A_{p}+A'_{p}$ is of degree less than $4m-3$. 

\item[3)] According to  Proposition (\ref{proposition/bornefraction}) and Theorem (\ref{theoreme/magnus}), we have $A_{p}(x) = a_{p}^{2m-1} \tilde{A}_{p}(x)$ and $B_{p}(x) = a_{p}^{2m-2} \tilde{B}_{p}(x)$  with $\tilde{A}_{p}$ and  $\tilde{B}_{p}$  polynomials  whose coefficients are bounded with respect to $p$. Hence the third point is proven.
\end{enumerate}

This conclude the base case.\ \\

\textbf{Induction step :} Let us assume that the points 1, 2  and 3 are true for a fixed $j\in\mathbb{N}$.

By computing the derivative, we get that

\begin{eqnarray*}
(\hat{L}_{p}(f))^{j+1} & = &  \dfrac{Q_{0,j,p}' A_{p}^{2^{j+1}}-2^{j+1}A_{p}'A_{p}^{2^{j+1}-1}Q_{0,j,p}}{A_{p}^{2^{j+2}}} f  \\
& + & \sum_{k=1}^{j+1}  \dfrac{Q_{k,j,p}' A_{p}^{2^{j+1-k}}-2^{j+1-k}A_{p}'A_{p}^{2^{j+1-k}-1}Q_{k,j,p} + Q_{k-1,j,p}}{A_{p}^{2^{j+2-k}}}      f^{(k)} \\
& + & \dfrac{Q_{j+1,j,p}}{A_{p}}   f^{(j+2)}.
\end{eqnarray*}

\begin{enumerate}
\item[1)] The point 1 is readily checked by using the induction assumption.

\item[2)] By the induction hypothesis and  the study of $B_{p}$ and $A_{p}$, we have that

$$
deg(Q_{k,j,p} ) \leq 2^{j+1-k}(2m-2)+1,
$$
$$
deg(Q_{k,j,p}' ) \leq 2^{j+1-k}(2m-2),
$$
$$
deg(Q_{k-1,j,p} ) \leq 2^{j+2-k}(2m-2)+1,
$$
$$
deg(A_{p}^{2^{j+1-k}}) = (2m-2)2^{j+1-k},
$$
$$
deg(A_{p}^{2^{j+1-k}-1}) = (2m-2)(2^{j+1-k}-1),
$$
$$
deg(A_{p}^{'}) = 2m-3.
$$

Therefore,

$$
deg(Q'_{k,j,p} A_{p}^{2^{j+1-k}} )\leq  2^{j+2-k}(2m-2) ,
$$

$$
deg( A_{p}^{'}A_{p}^{2^{j+1-k}-1} Q_{k,j,p}) \leq    (2m-2)2^{j+2-k} .
$$

The second point is thus true.

\item[3)] According to the induction assumption and  Proposition \ref{proposition/bornefraction}, we have the following equalities (where $\tilde{.}$ denotes a polynomial whose coefficients are bounded with respect to $p$): 

$$
Q_{k,j,p} = a_{p}^{2^{j+1-k}(2m-1)-1} \tilde{Q}_{k,j,p},
$$
$$
Q_{k-1,j,p} = a_{p}^{2^{j+2-k}(2m-1)-1} \tilde{Q}_{k-1,j,p},
$$
$$
Q'_{k,j,p} = a_{p}^{2^{j+1-k}(2m-1)-1} \tilde{Q'}_{k,j,p},
$$
$$
A_{p}^{2^{j+1-k}} = a_{p}^{(2m-1)2^{j+1-k}} \tilde{A}_{p}^{2^{j+1-k}},
$$
$$
A_{p}^{'} = a_{p}^{2m-3} \tilde{A'}_{p},
$$
$$
A_{p}^{2^{j+1-k}-1} = a_{p}^{(2m-1)(2^{j+1-k}-1)} \tilde{A}_{p}^{2^{j+1-k}-1} .
$$

Therefore,

$$
Q'_{k,j,p} A_{p}^{2^{j+1-k}} = a_{p}^{2^{j+2-k}(2m-1)-1}   \tilde{Q'}_{k,j,p}\tilde{A}_{p}^{2^{j+1-k}},
$$

$$
A_{p}^{'}A_{p}^{2^{j+1-k}-1} Q_{k,j,p} = a_{p}^{2^{j+2-k}(2m-1)-3} \tilde{A'}_{p} \tilde{Q}_{k,j,p}  \tilde{A}_{p}^{2^{j+1-k}-1} .
$$

These expressions satisfy the third point.\ \\

\end{enumerate}
\textbf{Conclusion : } By induction, the points 1, 2 and 3 are true for all $j\in\mathbb{N}$.

\end{proof}

\begin{remark}
The general case is very different of the Hermite case. Indeed, for the Hermite case, the coefficients are polynomials whose degrees grow linearly with the degree of derivation. In the general setting, the coefficients are rational fractions and the degree of their numerators grows exponentially with the degree of derivation. We think that the bottelneck of the proof is to study the fine properties of these rational fractions.
\end{remark}
Now that we have informations on $(\hat{L}_{p}(f))^{(j)}$, we are able to bound its $L^{2}(\rho)$ norm.

\begin{proposition}
Let $p\in\mathbb{N}$ be large enough so that $A_{p}$ is positive. For all $j\in\mathbb{N}$, there exists $K_{\phi,j}>0$ such that

$$
\| (\hat{L}_{p}(f))^{(j)} \|_{L^{2}(\rho)} \leq \dfrac{K_{\phi,j}}{a_{p}} \|f \|_{H^{j+1}(\rho )}.
$$
\label{proposition/estimation_derivee}
\end{proposition}

\begin{proof}
We use the triangle inequality and  Proposition \ref{proposition/formedescoeff}:

\begin{eqnarray*}
\| (\hat{L}_{p}(f))^{(j)} \|_{L^{2}(\rho)} & \leq & \sum_{k=0}^{j+1} \left\| \dfrac{Q_{k,j,p}}{A_{p}^{2^{j+1-k}}} f^{(k)}  \right\|_{L^{2}(\rho)} \\
 & \leq &  \sum_{k=0}^{j+1} \dfrac{ a_{p}^{2^{j+1-k}(2m-1)-1} }{{A_{p}^{2^{j+1-k}}}(0)} \| \tilde{Q}_{k,j}   f^{(k)} \|_{L^{2}(\rho)} .
 \end{eqnarray*}
 
Now, we use the generalized Strong Poincaré inequality (\ref{proposition/poincaregeneralise}) with $r=2^{j+1}+1$, and  Proposition \ref{proposition/bornefraction}. This gives that there exists a constant $ K_{\phi,j}$ depending only on $j,\phi$ such that:

\begin{equation}
\| (\hat{L}_{p}(f))^{(j)} \|_{L^{2}(\rho)}  \leq   \dfrac{ K_{\phi,j}}{a_{p}} \sum_{k=0}^{j+1} \| f^{(k)} \|_{H^{2^{j+1}+1}(\rho)}
 \leq    \dfrac{ K_{\phi,j}}{a_{p}}  \| f \|_{H^{2^{j+1}+j+2}(\rho)}.
 \end{equation}

\end{proof}

Now that the $L^{2}(\rho)$ norm of $(\hat{L}_{p}(f))^{(j)}$ is estimated, these estimates can be summed to give an estimation of the Sobolev norm.

\begin{proposition}
Let $p\in\mathbb{N}$ be large enough so that $A_{p}$ is positive. Let $r\in\mathbb{N}$. There exists a constant $\theta_{\phi,r}>0$ such that for all $f\in C^{\infty}_{c}(\mathbb{R})$,

$$
\| \hat{L}_{p}(f) \|_{H^{r}(\rho)} \leq \dfrac{\theta_{\phi,r}}{a_{p}} \| f \|_{H^{2^{r+1}+r+2}(\rho)}.
$$
\label{proposition/estimationnorme}
\end{proposition}

\begin{proof}
The proof relies on a direct computation allowed by the  Proposition \ref{proposition/estimation_derivee}:
\begin{eqnarray*}
\| \hat{L}_{p}(f) \|_{H^{r}(\rho)} & = &  \sum_{j=0}^{r} \| (\hat{L}_{p}(f))^{(j)} \|_{L^{2}(\rho )} \\
& \leq &  \sum_{j=0}^{r} \dfrac{K_{\phi,j}}{a_{p}} \| f \|_{H^{2^{j+1}+j+2}(\rho)} \\
& \leq &  \dfrac{1}{a_{p}}  \left(\sum_{j=0}^{r} K_{\phi,j} \right) \| f \|_{H^{2^{r+1}+r+2}(\rho)}. \end{eqnarray*}
We conclude by defining  $\theta_{\phi,r} :=  \sum_{j=0}^{r} K_{\phi,j}$.
\end{proof}

Finally, we can prove the main result as stated in Theorem \ref{theoreme/mainresult}.

\begin{proof}[Proof of Theorem \ref{theoreme/mainresult}]

First,  Parseval identity implies:

$$
\| f - \pi_{V_{N}}f \|_{L^{2}(\rho)}^{2} = \sum_{n=N+1}^{\infty}\left\lvert \int_{\mathbb{R}} f \tilde{P}_{n} \rho dx \right\lvert^{2}.
$$

Then, we apply integrations by part, Cauchy Schwarz inequality and the submutiplicativity of the operator norms. For $k>m$ fixed, we have

\begin{eqnarray*}
\sum_{n=N+1}^{\infty}\left\lvert \int_{\mathbb{R}} f \tilde{P}_{n} \rho dx \right\lvert^{2} & \leq & \sum_{n=N+1}^{\infty}  \int_{\mathbb{R}} \left( \bigcircop_{j=1}^{k} \hat{L}_{n+k-j+1} \right)(f) \tilde{P}_{n+k} \rho dx  \\
& \leq &  \sum_{n=N+1}^{\infty} \prod_{j=1}^{k} \|\hat{L}_{n+k-j+1} \|^{2}_{H^{\gamma_{j}}(\rho )\to H^{\gamma_{j-1}}(\rho )} \| f\|^{2}_{H^{\gamma_{k}}(\rho )} \\
 & \leq &  \sum_{n=N+1}^{\infty}\prod_{j=1}^{k}   \dfrac{\theta_{\phi,j}^{2}}{a^{2}_{n+k-j+1}} \|f\|^{2}_{H^{\gamma_{k}}(\rho )} \\
 & \leq &  \|f\|^{2}_{H^{\gamma_{k}}(\rho )} \left(\prod_{j=1}^{k} \theta_{\phi,j}^{2}\right) \sum_{n=N+1}^{\infty}\left(\prod_{j=1}^{k}   \dfrac{a_{n}^{2}}{a_{n+k-j+1}^{2}}  \right)\dfrac{1}{a_{n}^{2k}}  .
\end{eqnarray*}

Finally, we use  Proposition \ref{proposition/estimationnorme}, Magnus's Theorem \ref{theoreme/magnus} and we get

\begin{eqnarray*}
\sum_{n=N+1}^{\infty}\left\lvert \int_{\mathbb{R}} f \tilde{P}_{n} \rho dx \right\lvert^{2} & \leq &  \sum_{n=N+1}^{\infty}\prod_{j=1}^{k}   \dfrac{\theta_{\phi,j}^{2}}{a^{2}_{n+k-j+1}} \|f\|^{2}_{H^{\gamma_{k}}(\rho )} \\
 & \leq &  \|f\|^{2}_{H^{\gamma_{k}}(\rho )} \left(\prod_{j=1}^{k} \theta_{\phi,j}^{2}\right) \sum_{n=N+1}^{\infty}\left(\prod_{j=1}^{k}   \dfrac{a_{n}^{2}}{a_{n+k-j+1}^{2}}  \right)\dfrac{1}{a_{n}^{2k}}  .
\end{eqnarray*}

Hence there exists a constant $\Lambda_{\phi,k}$ such that the projection error is bounded by

$$
\Lambda_{\phi,k} \| f\|_{H^{\gamma_{k}}(\rho)}^{2} \sum_{n=N+1}^{\infty}  \dfrac{1}{a_{n}^{2k}}.
$$

It is the queue of a series whose term is equivalent to the term of a Riemann series by Theorem \ref{theoreme/magnus}. This queue is finite if and only if $k>m$. Therefore, for all $f\in C^{\infty}_{c}(\mathbb{R})$,

$$
\| f - \pi_{V_{N}}f \|_{L^{2}(\rho)}^{2} \leq \Lambda_{\phi,k} \| f\|_{H^{\gamma_{k}}(\rho)}^{2} \sum_{n=N+1}^{\infty}  \dfrac{1}{n^{\frac{k}{m}}}.
$$

We can be more explicit and give the order of convergence of the right-hand side. As a byproduct of the integral test for convergence, we have that

$$
\sum_{n=N+1}^{\infty} \dfrac{1}{n^{\frac{k}{m}}} \leq \int_{N}^{\infty} \dfrac{1}{x^{\frac{k}{m}}} dx = \dfrac{m}{k-m} \dfrac{1}{N^{\frac{k}{m}-1}}.
$$

Perhaps modifying $\Lambda_{\phi,k}$, we get the bound

$$
\| f - \pi_{V_{N}}(f) \|_{L^{2}(\rho)}^{2} \leq \Lambda_{\phi,k} \| f\|_{H^{\gamma_{k}}(\rho)}^{2} \dfrac{1}{N^{\frac{k}{m}-1}} .
$$

By density of the $C^{\infty}_{c}$ functions in the Sobolev space $H^{\gamma_{k}}(\rho)$ and by continuity of the identity and of the projection $\pi_{V_{N}}$  of $L^{2}(\rho)$ in itself, we obtain the result for all $f\in H^{\gamma_{k}}(\rho)$. The proof of Theorem \ref{theoreme/mainresult} is complete.
\end{proof}

\newpage

\section{Numerical simulations}
\label{section/4}

\subsection{Generating orthonormal polynomials}
\label{subsection/4.1}

\subsubsection{The algorithms}

In general, the generation of a sequence of orthonormal polynomials is challenging. Indeed, the values of the recursion coefficients $a_{n}$ are only known exactly in a few cases, such as the Hermite case. In order to compute their values for a potential $\phi$ of degree strictly larger than $2$, we will apply the Chebychev algorithm detailed in \cite{gautschi}.

It is a moment-based method. If $\mu_{k}$ stands for the $k$-th moment of the measure $\rho(x) dx$, then for $n\in\mathbb{N}^{*}$ fixed, the algorithm requires the moments $(\mu_{k})_{0\leq k\leq 2n-1}$ and returns $(\beta_{k})_{0\leq k\leq n-1}= (a^{2}_{k})_{0\leq k \leq n-1}$. The  algorithm is the following.

\begin{algo}[Chebychev algorithm \cite{gautschi}]
\ \\

\textbf{Initialisation :}

$$\beta_{0} = \mu_{0}$$

$$\sigma_{-1,l} = 0 \ \mbox{ for } \  l=1,...,2n-2$$ 

$$\sigma_{0,l} = \mu_{l} \ \mbox{ for } \ l=0,...,2n-1$$ 
\ \\

\textbf{For $k=1,..,n-1$ do : }

$$\sigma_{k,l} = \sigma_{k-1,l+1} - \beta_{k-1}\sigma_{k-2,l} \ \mbox{ for } \ l=k,..,2n-k-1  $$

$$\beta_{k} = \dfrac{\sigma_{k,k}}{\sigma_{k-1,k-1}}$$

\textbf{End do}
\ \\

\textbf{Return} $(\beta_{k})_{0\leq k \leq n-1} $
\end{algo}

As our measure $\rho(x)dx$ is even, only  even moments are non-nul and there are some simplifications we don't explicit here.

Two difficulties arise. The first one is that, according to \cite{gautschi}, the algorithm is ill-conditionned. The error can grow fast as $k$ increases. This leads us to the second difficulty: the moments must be computed with a high precision in order to get an exploitable result. A first <<brut-force>> approach would be to compute all the required moments by using a very accurate quadrature method. A second approach is based upon the following  induction relation satisfied by the moments $\mu_{k}$ of $\rho(x)dx$.

\begin{proposition}
Let $\phi(x) = \sum_{p=0}^{m}v_{p}x^{2p}$ be a even polynomial with $v_{m}>0$. For all $k\in\mathbb{N}$, 

$$
\left\lbrace
\begin{array}{ccc}
2mv_{m} \mu_{2(m+k)} & = & (2k+1)\mu_{2k} - \sum_{p=1}^{m-1} 2pv_{p} \mu_{2(p+k)} ,\\
\mu_{2k+1} =   0. & & 
\end{array}
\right.
$$

\end{proposition}

\begin{proof}
All the odd moments are zero due to the fact that $\rho = e^{-\phi}$ is even. Next, let us choose an even moment $\mu_{2k}$ with $k\in\mathbb{N}$. By integration by parts:

\begin{eqnarray*}
\mu_{2k} & = & \int_{\mathbb{R}} t^{2k}  e^{-\phi(t)} dt \\
		 & = & \left[ \dfrac{t^{2k+1}}{2k+1} e^{-\phi(t)} \right]_{-\infty}^{+\infty} + \int_{\mathbb{R}} \dfrac{t^{2k+1}}{2k+1} \phi'(t) e^{-\phi(t)} dt \\
		 & = & \sum_{p=1}^{m} \dfrac{2pv_{p}}{2k+1} \int_{\mathbb{R}} t^{2k+2p} e^{-\phi(t)} dt \\
		 & = & \sum_{p=1}^{m} \dfrac{2pv_{p}}{2k+1}\mu_{2k+2p}. \\
\end{eqnarray*}

By rearanging the terms, we find that $2mv_{m} \mu_{2(m+k)}  = (2k+1)\mu_{2k} - \sum_{p=1}^{m-1} 2pv_{p} \mu_{2(p+k)}$.
\end{proof}

The aformentionned relation allows one to compute all the moments, provided that the first even moments $\mu_{0},\mu_{2},...,\mu_{2m-2}$ are known. These moments can be computed by using a high-order quadrature method. In the Hermite case, the recurrence relation provides an explicit  expression for $\mu_{2k}$:

$$
\mu_{2k} = \dfrac{(2k)!}{2^{k}k!}.
$$

Stirling's approximation stays then that $\mu_{2k} \sim \sqrt{2} \dfrac{k^{k}}{(2e)^{k}}$, hence growing fast when $k\to \infty$. This may lead to errors if the significand precision used in the implementation isn't large enough to store all the digits. Because of this issue and the fact that the Chebychev algorithm is ill-conditionned, we will use the quadruple precision floating point format included in Fortran.

\subsubsection{Comparison of the two methods} 

First, we compare the results given by the brute-force method  and by the induction method on the Hermite case. For this case ,  $\mu_{2k} = \dfrac{(2k)!}{2^{k}k!}$ and $\beta_{k}=k$.

All the numerical integration are done by using a composite Weddle-Hardy quadrature on the interval $[-30,30]$ with 6*350 points. The results are in Figure \ref{figure/erreur_construction_polynome}.

We chosed to represent the $L^{2}(\rho)$ error between exact Hermite polynomials and there approximation  and the evolution of the error on the coefficients $\beta_{k}$. It seems that the two methods give similar results, although the recurrence based-method behaves a little better.

Then, we compare the results given by the two methods on the double well case which we recall is the case when $\phi(x)=(x-1)^{2}(x+1)^{2}$. This time, the theoretical values are unknown. At least, Magnus's Theorem \ref{theoreme/magnus} gives the asymptotic behaviour of $\sqrt{\beta_{n}}$. We can then test if the sequence generated by the Chebychev algorithm verify this asymptotic (see Figure \ref{figure/beta}).

The recurrence-based method is unstable, and diverges for $N\approx 30$. For larger  $N$, $\beta_{N}$ is even  negative. The recurrence relation is too sensitive to initial conditions and propagates the error quickly. The brute-force method seems more stable, as it starts to diverge for $N\geq 60$.  In the following, we will only use the brute-force method which is more costly but is more stable.

\begin{figure}[h!]
\centering
\includegraphics[scale=0.5]{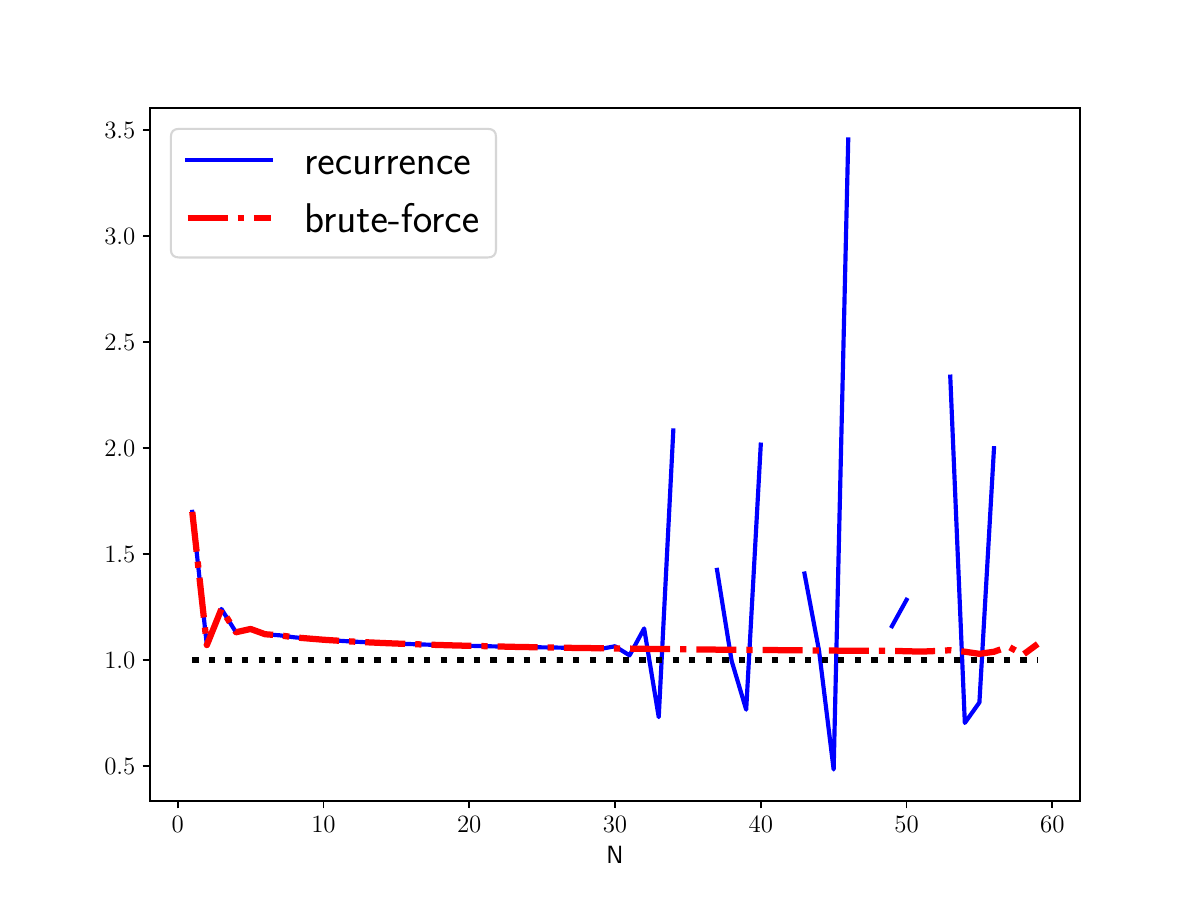}
\caption{Asymptotic behavior of the recursion coefficients $\beta_{k}$ computed by the Chebychev algorithm, for the brute-foce and the recursion-based method.}
\label{figure/beta}
\end{figure}

\begin{figure}[h!]
\centering

\begin{subfigure}{\textwidth}
\centering
\begin{tabular}{|c|c|c|}
\hline
$N$    &   brute-force  & recurrence-based \\
\hline
10	& $1.5 *10^{-30}$ & $1.8* 10^{-31}$ \\
\hline
20	& $9.0 *10^{-24}$ & $4.4* 10^{-26}$ \\
\hline 
30	& $2.1 *10^{-20}$ &  $2.3* 10^{-22}$ \\
\hline
40	& $9.6 *10^{-16}$	& $2.7* 10^{-16}$ \\
\hline
50	& $1.8 *10^{-10}$	& $2.3* 10^{-11}$ \\
\hline
60	& $4.7 *10^{-6}$ & $1.1* 10^{-6}$ \\
\hline
70	& $0.27$    & $0.13$ \\
\hline
\end{tabular}
\subcaption{Evolution of the error $\max_{k\leq N} |\beta_{k}-k|$ on the coefficients $\beta_{k}$.}
\end{subfigure}

\begin{subfigure}{\textwidth}
\centering
\includegraphics[width=0.6\textwidth]{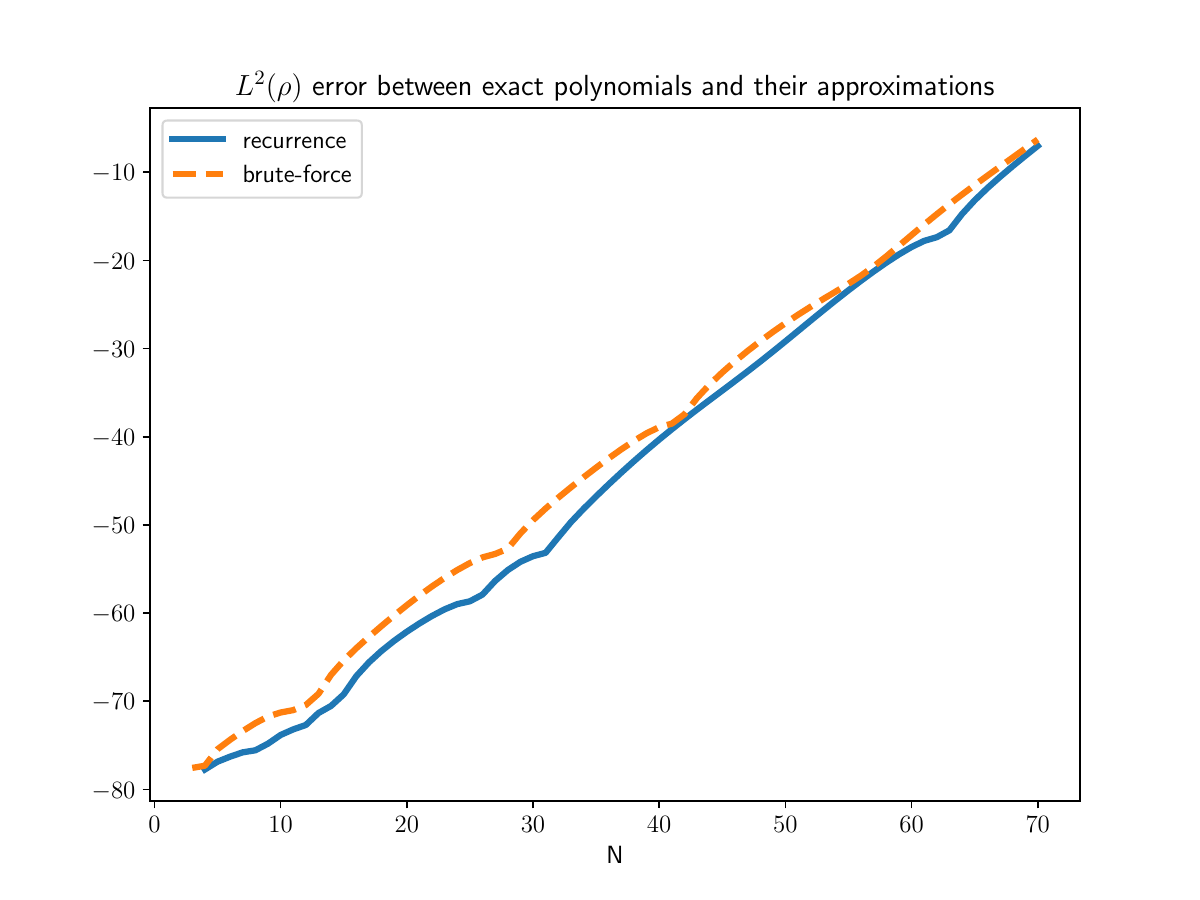}
\subcaption{Evolution of the $L^{2}(\rho)$ error between exact and approximate Hermite polynomials. The error is in log scale.}
\end{subfigure}
\caption{Comparison of the brute-force and of the recurrence-based methods in the Hermite case.}
\label{figure/erreur_construction_polynome}

\end{figure}

\newpage
\ \\
\newpage
\subsection{Approximation of functions by orthonormal polynomials}
\label{subsection/4.2}
Section \ref{subsection/4.1} addressed the computation of  orthonormal polynomials with respect to the measure $\rho(x)dx$. Now, we evaluate the projection error of a function $f$ for several truncation rank $N$. Since the order of convergence should depend on the regularity of $f$, we will approximate  different smooth and non-smooth functions and compare the experimental orders of convergence with the theoretical ones.

The three smooth tested functions  are 

$$
f(x) = \exp(x) ; \ \ \ g(x) = \exp(0.1*x^{2}) ; \ \ \ h(x) = \cos(x).
$$

The four non-smooth  tested functions are 

$$
f_{k}(x) = 
\left\lbrace
\begin{array}{cc}
0 & \mbox{ if } x<0 \\
x^{k} & \mbox{ if } x\geq 0
\end{array}
\right. \ \ k=1,2,3,4.
$$

Each $f_{k}$ belongs to $H^{k}(\rho)$ but not to $H^{k+1}(\rho)$.

\subsubsection{Hermite case:  $\phi(x) = \frac{1}{2}(x^{2}+\ln (2\pi))$}

For the smooth functions, the projection error decays very fast (figure (\ref{figure/m1_projection_error})), as the graph are not straight lines. For $N\geq 30$ the error equals at most $10^{-9}$. 

For non-smooth functions, the projection error decays polynomialy (figure (\ref{figure/m1_projection_error})). The order of convergence increases as the regularity $k$ increases. In fact, Proposition (\ref{proposition/hermitecase}) ensures an order at least equals to $\frac{k}{2}$. Experimental orders agree with this result, and  show that the value $\frac{k}{2}$ is nearly optimal.

\begin{figure}[h!]
\centering
\begin{subfigure}{\textwidth}
\centering
\includegraphics[width=0.8\textwidth]{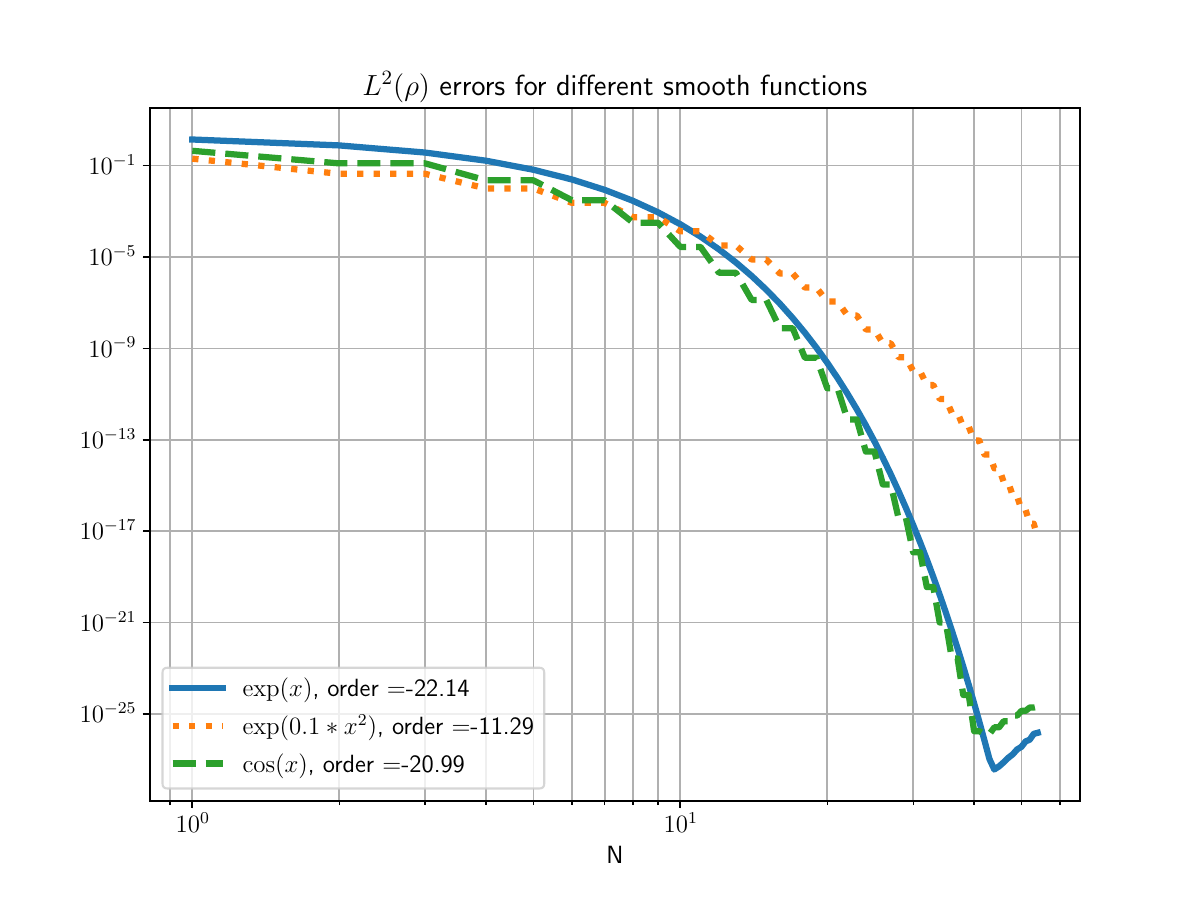}
\subcaption{Smooth functions}
\end{subfigure}
\begin{subfigure}{\textwidth}
\centering 
\includegraphics[width=0.8\textwidth]{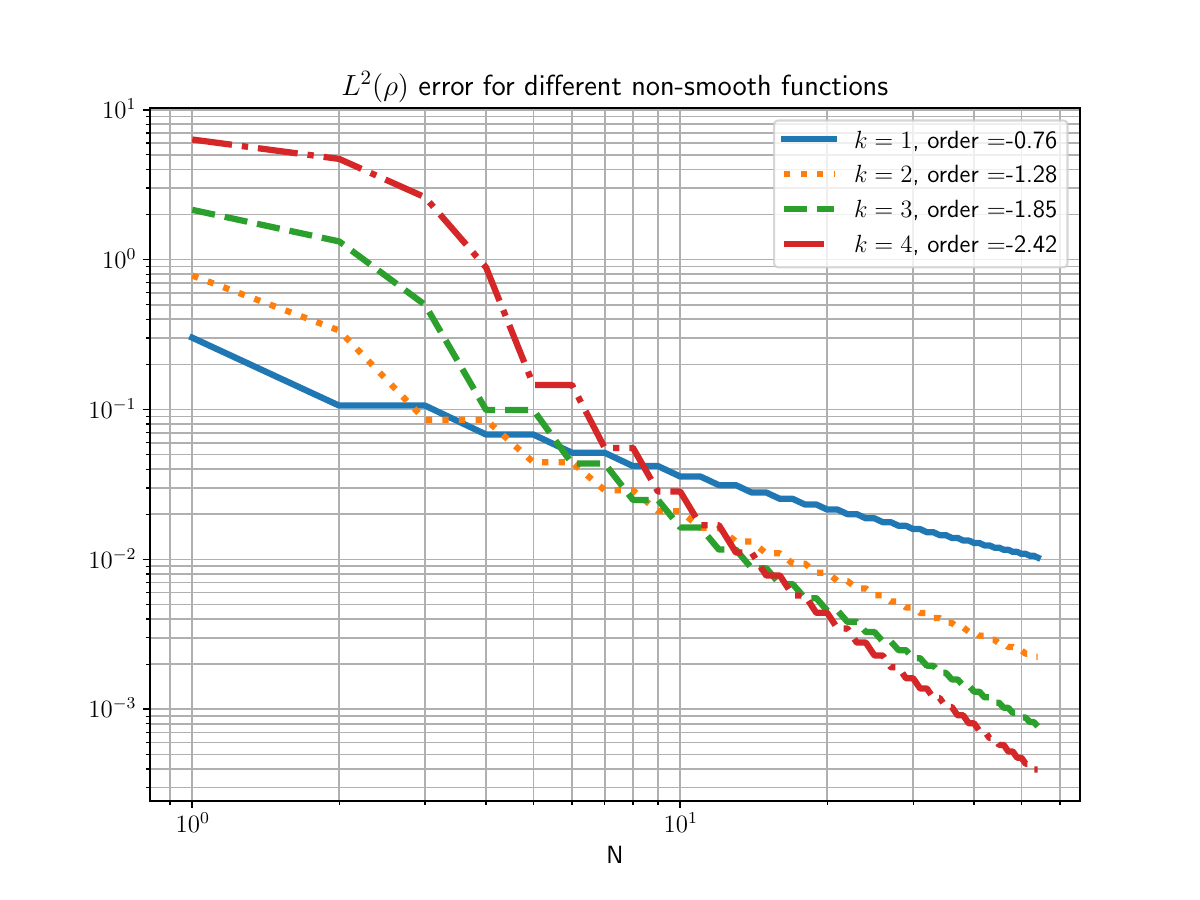}
\subcaption{Non-smooth functions}
\end{subfigure}
\caption{Evolution of the projection error. The $y$-axis is in log scale. Orders of convergence are displayed in the legend.}
\label{figure/m1_projection_error}
\end{figure}

\subsubsection{Double well case:  $\phi(x) =(x-1)^{2}(x+1)^{2}$}

Although our main result is useless in practice due to the exponentially big constants,  we run tests on the double well case.
For the smooth functions, the observations are the same as in the Hermite case (Figure \ref{figure/m2_projection_error}).

For non-smooth functions, the projection error decays polynomialy (Figure \ref{figure/m2_projection_error}). The order of convergence increases as the regularity $k$ increases.

\begin{figure}[h!]
\centering
\begin{subfigure}{\textwidth}
\centering
\includegraphics[width=0.8\textwidth]{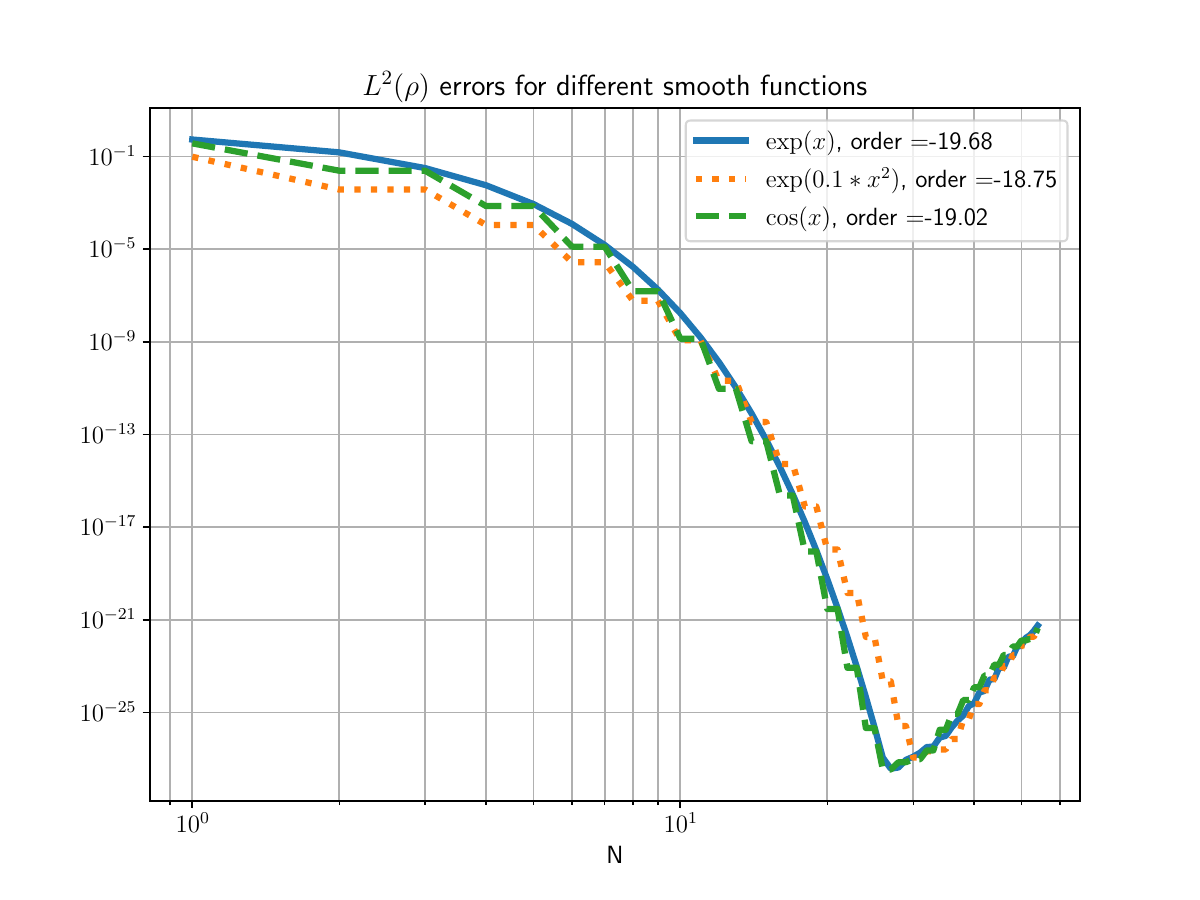}
\subcaption{Smooth functions}
\end{subfigure}
\begin{subfigure}{\textwidth}
\centering 
\includegraphics[width=0.8\textwidth]{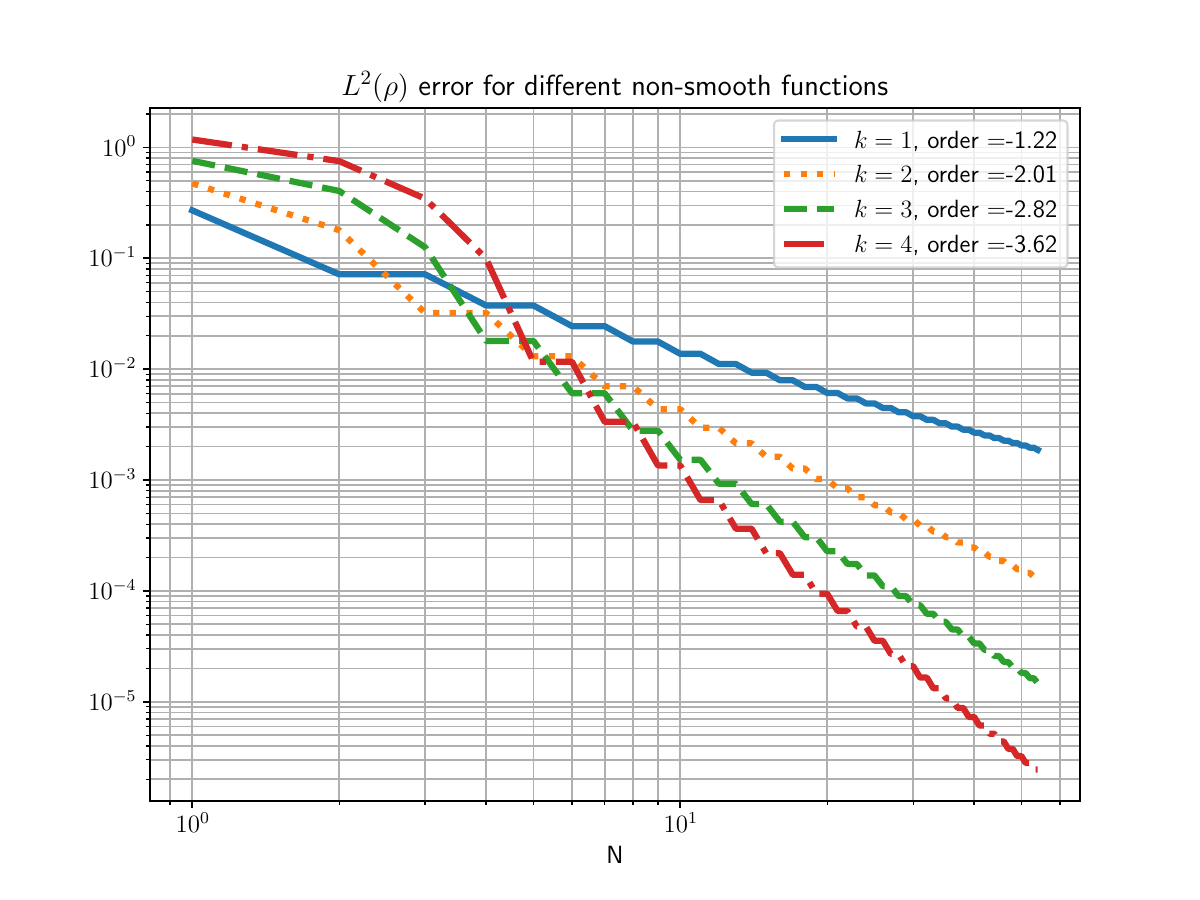}
\subcaption{Non-smooth functions}
\end{subfigure}
\caption{Evolution of the projection error. The $y$-axis is in log scale. Orders of convergence are displayed in the legend.}
\label{figure/m2_projection_error}
\end{figure}
\newpage
\ \\
\ \\
\ \\
\ \\

\section{Remarks on the proof of the main theorem}

Our proof is hard to use in practice due to the weak  bounds of the rational fraction $\left|\dfrac{Q_{k,j,p}(x)}{A_{p}^{2^{j+1-k}}(x)}\right|$ by $\left|\dfrac{Q_{k,j,p}(x)}{A_{p}^{2^{j+1-k}}(0)}\right|$   in Proposition \ref{proposition/estimation_derivee}. As the degree of the polynomial $Q_{k,j,p}$ grows rapidly with $k$, we need to apply the Strong Poincaré inequality which only works for  very  regular functions. 

Remember that the operator in the general case reads
$$
\hat{L}_{n}(f) = - \dfrac{1}{A_{n}} f' +  \dfrac{B_{n}A_{n}+\phi'A_{n}+A'_{n}}{A_{n}^{2}} f  .
$$ 

To get a better insight on the problem, we place ourselves in the double well case. We draw the graphs of the two rationnal functions $G_{n}:=\frac{1}{A_{n}}$, $F_{n}:=\frac{B_{n}A_{n}+\phi'A_{n}+A'_{n}}{A_{n}^{2}}$ and its asymptote $\frac{x}{a_{n}}$ for several integers $n$ (Figure \ref{figure/n}). We see that $F_{n}$ becomes closer to its asymptote as  $n$ becomes larger. We may write 

$$
F_{n}(x) = \dfrac{x}{a_{n}} + R_{n}(x)
$$ 

with $R_{n}$ the remainder term. Hence, $\hat{L}_{n}(f)$ becomes

$$
\hat{L}_{n}(f) = - \dfrac{1}{A_{n}} f' + \dfrac{x}{a_{n}} f + R_{n} f  .
$$

Let us  now try to estimate the projection error with this new identity. By using successively  Parseval formula, Cauchy-Schwarz inequality, Proposition \ref{proposition/bornefraction} and the Strong Poincaré inequality, we get :

\begin{eqnarray*}
\| f - \pi_{V_{N}}f \|_{L^{2}(\rho)}^{2} & = & \sum_{n=N+1}^{\infty} \left|\int_{\mathbb{R}} f \tilde{P}_{n} \rho dx \right|^{2} \\
 & = &  \sum_{n=N+1}^{\infty}  \left|\int_{\mathbb{R}} \hat{L}_{n+1}(f)  \tilde{P}_{n+1} \rho dx \right|^{2} \\
 & \leq &   \sum_{n=N+1}^{\infty}  \left|\int_{\mathbb{R}}( \frac{1}{A_{n+1}}f' + \frac{x}{a_{n+1}}f + R_{n+1} f) \tilde{P}_{n+1} \rho dx \right|^{2}  \\
 & \leq &  3 \sum_{n=N+1}^{\infty}   \left|\int_{\mathbb{R}} \frac{1}{A_{n+1}}f' \tilde{P}_{n+1} \rho dx \right|^{2} + \left|\int_{\mathbb{R}} \frac{x}{a_{n+1}}f  \tilde{P}_{n+1}\rho dx \right|^{2} + \left|\int_{\mathbb{R}} R_{n+1} f \tilde{P}_{n+1} \rho dx \right|^{2} \\
 & \leq &  3 \sum_{n=N+1}^{\infty}   \left\| \frac{1}{A_{n+1}}f'\right\|^{2} + \left|\int_{\mathbb{R}} \frac{x}{a_{n+1}}f  \tilde{P}_{n+1}\rho dx \right|^{2} + \left|\int_{\mathbb{R}} R_{n+1} f \tilde{P}_{n+1} \rho dx \right|^{2} \\
 & \leq &  3 \left\| f'\right\|^{2} \sum_{n=N+1}^{\infty}  \frac{1}{A_{n+1}(0)^{2}}  + 3\dfrac{1}{a_{N+2}^{2}} \|xf\|^{2} + 3  \sum_{n=N+1}^{\infty}  \left|\int_{\mathbb{R}} R_{n+1} f \tilde{P}_{n+1} \rho dx \right|^{2} \\
  & \leq &  3 \left\| f'\right\|^{2} \sum_{n=N+1}^{\infty}  \frac{1}{A_{n+1}(0)^{2}}  + 3 C_{p}^{2}\dfrac{1}{a_{N+2}^{2}} \|f\|^{2}_{H^{1}(\rho)} + 3  \sum_{n=N+1}^{\infty}  \left|\int_{\mathbb{R}} R_{n+1} f \tilde{P}_{n+1} \rho dx \right|^{2} .
\end{eqnarray*}

Here, we have used the classic Strong Poincaré inequality for $H^{1}(\rho)$ functions. However, we don't see how to estimate the last series, as  we lack  estimates on these integrals. A better comprehension of this last series may lead to a more precise approximation theorem than Theorem \ref{theoreme/mainresult}.  

\begin{figure}[h!]

\begin{subfigure}{0.45\textwidth}

\includegraphics[width=\textwidth]{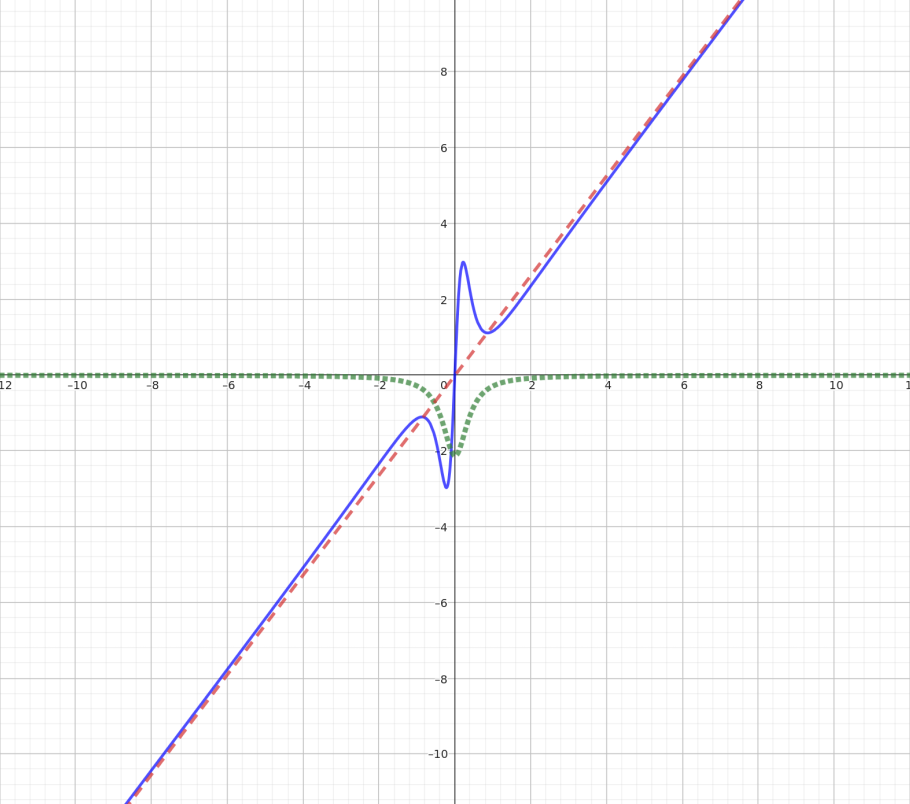}
\subcaption{$n=4$}
\end{subfigure}
\hfill
\begin{subfigure}{0.45\textwidth}
 
\includegraphics[width=\textwidth]{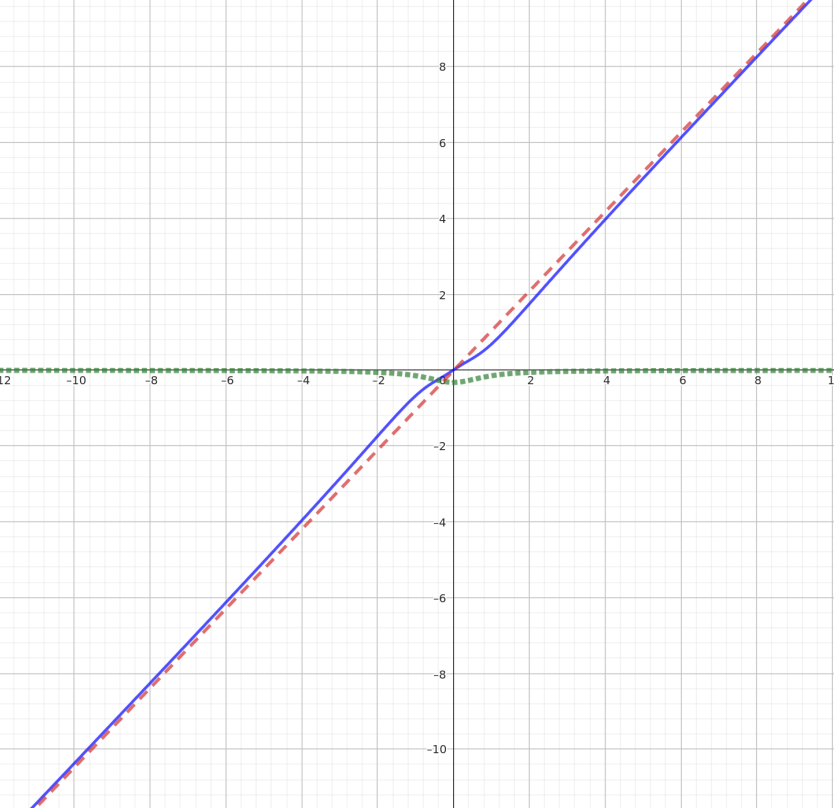}
\subcaption{$n=10$}
\end{subfigure}
\begin{subfigure}{0.6\textwidth}
\centering 
\includegraphics[width=1.2\textwidth]{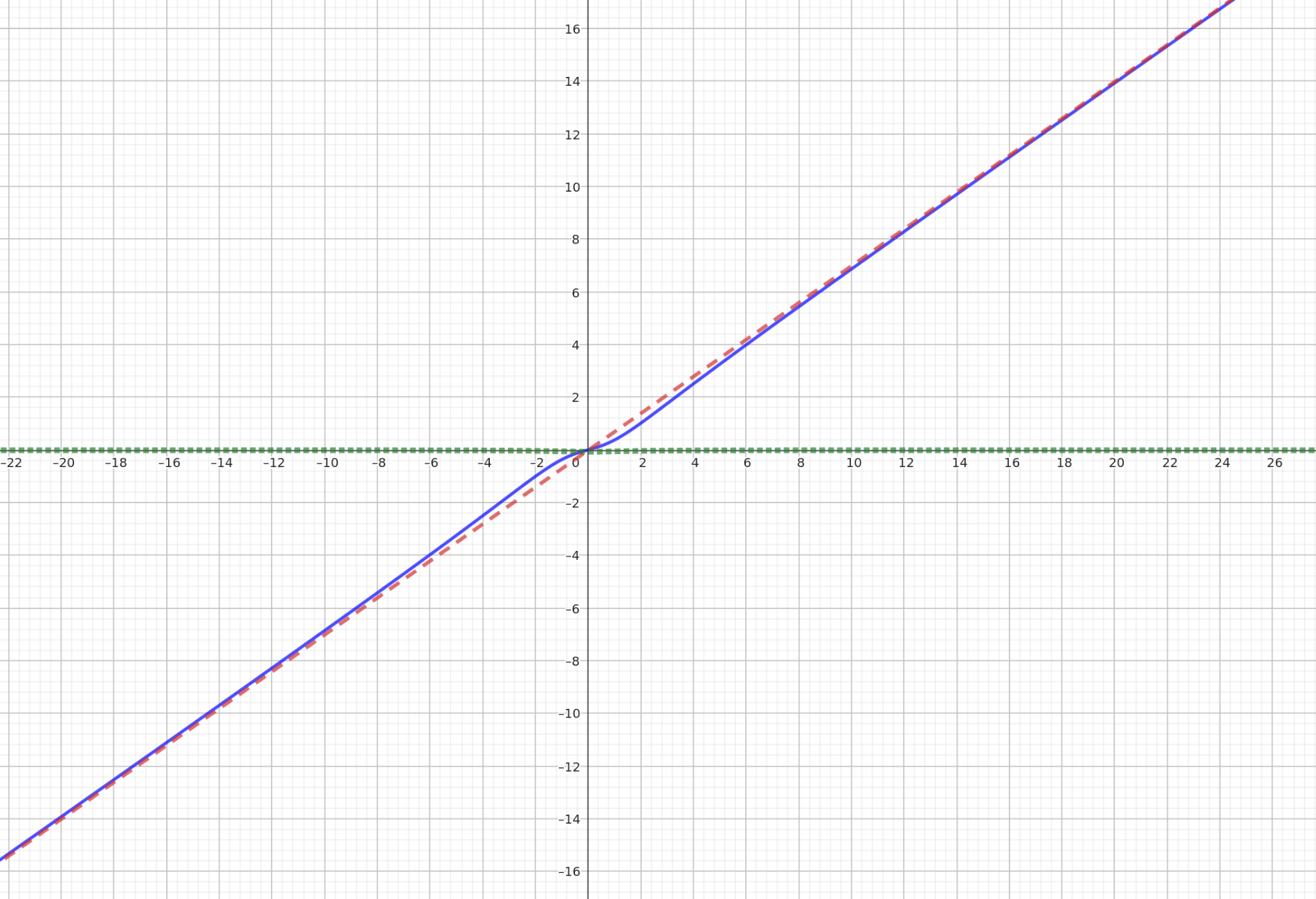}
\subcaption{$n=50$}
\end{subfigure}
\caption{Continuous line: $F_{n}$. Thick dotted line: $G_{n}$. The asymptote of $F_{n}$ is the straight dotted line.}
\label{figure/n}
\end{figure}

\ \\
\ \\
\ \\
\ \\
\ \\
\ \\
\ \\
\ \\
\ \\
\ \\
\ \\
\ \\
\ \\
\ \\
\ \\
\ \\
\ \\
\ \\

\bibliographystyle{siam}
\bibliography{biblio.bib}

\end{document}